\documentclass[12pt]{amsart}
\usepackage{a4wide}
\usepackage{amsmath,amssymb,amsthm}
\usepackage{url}
\usepackage{hyperref}
\usepackage{graphics}
\usepackage{graphicx}
\usepackage{bm}
\usepackage{pdfsync}
\usepackage{enumerate}
\DeclareMathOperator{\sgn}{sgn}

\usepackage{mathtools}
\mathtoolsset{showonlyrefs}

\numberwithin{equation}{section}

\usepackage[normalem]{ulem}

\newcommand{\Hm}[1]{\leavevmode{\marginpar{\tiny
			$\hbox to 0mm{\hspace*{-0.5mm}$\leftarrow$\hss}
			\vcenter{\vrule depth 0.1mm height 0.1mm width \the\marginparwidth}
			\hbox to			0mm{\hss$\rightarrow$\hspace*{-0.5mm}}$\\\relax\raggedright #1}}}

\usepackage[dvipsnames]{xcolor}
\usepackage{color}
\newtheorem{theorem}{Theorem}[section]
\newtheorem{proposition}[theorem]{Proposition}
\newtheorem{remark}[theorem]{Remark}
\newtheorem{lemma}[theorem]{Lemma}
\newtheorem{corollary}[theorem]{Corollary}
\newtheorem{definition}[theorem]{Definition}

\def\bu{{\bf{u}}}
\def\b0{{\bf 0}}
\def\bw{{\bf{w}}}
\def\bv{{\bf{v}}}
\def\bpsi{{\bm \psi}}
\def\bp{{\bm \varphi}}
\def\boldeta{{\bm \eta}}

\def\eps{\varepsilon}

\def\dg{\Delta_\Gamma}

\def\rr{\mathbb{R}}
\def\to{\rightarrow}

\def\um{\underline{M}}
\def\om{\overline{M}}

\begin{document}

\title[Convection-diffusion on star shaped graphs]{A convection-diffusion model on a star-shaped graph}

\author[C. M.  Cazacu]{Cristian M. Cazacu}
\address[C. M. Cazacu]{Faculty of Mathematics and Computer Science \& The Research Institute of the University of Bucharest (ICUB),  University of Bucharest\\
14 Academiei Street \\ 010014 Bucharest\\ Romania
\& Gheorghe Mihoc-Caius Iacob Institute of Mathematical
Statistics and Applied Mathematics of the Romanian Academy\\
050711 Bucharest, Romania}
\email{cristian.cazacu@fmi.unibuc.ro}

\author[L. I. Ignat]{Liviu I. Ignat}
\address[L. I. Ignat]{Institute of Mathematics ``Simion Stoilow'' of the Romanian Academy,
Centre Francophone en Math\'{e}matique\\
21 Calea Grivitei Street \\010702 Bucharest \\ Romania\\
\hfill\break\indent \and
\hfill\break\indent
  The Research Institute of the University of Bucharest (ICUB), University of Bucharest\\
90-92 Sos. Panduri, 5th District, Bucharest, Romania\\
}

\email{liviu.ignat@gmail.com}

\author{Ademir F. Pazoto}
\address[A. F. Pazoto]{Instituto de Matem\'atica, Universidade Federal do Rio de
Janeiro, P.O. Box 68530, CEP 21945-970, Rio de Janeiro, RJ,  Brasil}
\email{ademir@im.ufrj.br}

\author{Julio D. Rossi}
\address[J. D. Rossi]{
Dpto. de Matem{\'a}ticas, FCEyN,
Universidad de Buenos Aires, 1428, Buenos Aires,
Argentina. } \email{{\tt jrossi@dm.uba.ar} }

\begin{abstract}

 In this paper we study a convection-diffusion equation on a star-shaped graph composed by $n$ incoming edges and $m$ outgoing edges with a  nonlinearity $f\in C^1(\rr)$ satisfying some additional general conditions.  First, we prove the global well-posedness of the solutions of the system under consideration. Next, in the particular case that the nonlinear convection is given by $\partial_x(f(u(t, x))$ with
  $f(s)=-a|s|^{q-1}s$ with $q\geq 2$ and $a\in \rr$ verifying $(n-m)a\geq 0$, we analyze the long time behavior of the solutions.
 For $q> 2$ we find that the asymptotic behavior of the solutions is given by some self-similar profiles of the heat equation on the considered structure.
 In the case $q=2$, the nonnegative/nonpositive solutions converge to the self-similar profiles of Burgers' equation.
 Explicit representations of the limit profiles are obtained.
 \end{abstract}

\maketitle



\section{Introduction} \label{sect-intro}

The study of well-posedness and asymptotic behavior of time-evolution nonlinear partial differential equations (PDE) is a classical
and relevant topic which has been extensively studied in the last decades.
A list of related references is quite large and difficult to describe with precision in few lines.
We  refer to \cite{MR2656972, Karch2011} for reviews of the methods used to obtain the asymptotic profiles.

In this paper, we consider a PDE with linear diffusion and a nonlinear convective term in the ambient space of a network formed by the edges of a graph. Our goal is to prove the global well-posedness of the problem and to determine the asymptotic behavior
of the solutions for large times.  The bibliography on the study of equations on networks is also very vast. For asymptotic spectral analysis and further motivations of such equations we refer e.g. to \cite{MR2169126,MR3013208} and the references therein.

To be more specific, here we deal with a convection-diffusion model on a simple 1-d network (denoted by $\Gamma$)
formed by the edges of a star-shaped graph, which is composed by a
single junction with $n$ incoming infinite edges and $m$ outgoing infinite edges.
From the mathematical point of view we describe our network $\Gamma$ as  follows: the junction of the edges is settled  at $x=0$;  the $n$ incoming edges, indexed by $i\in \{1,\ldots,n\}$, are parameterized by the negative real axis $I_i=\mathbb{R}_{-}:=(-\infty,0]$ whereas the set of the $m$ outgoing edges, indexed by $j\in \{1, \ldots, m\}$,   are parametrized by the positive real axis $I_j=\mathbb{R}_{+}:=[0,\infty)$ (see Figure \ref{fig1}).
\begin{figure}[htbp]
	\begin{center}
		\includegraphics[scale=1]{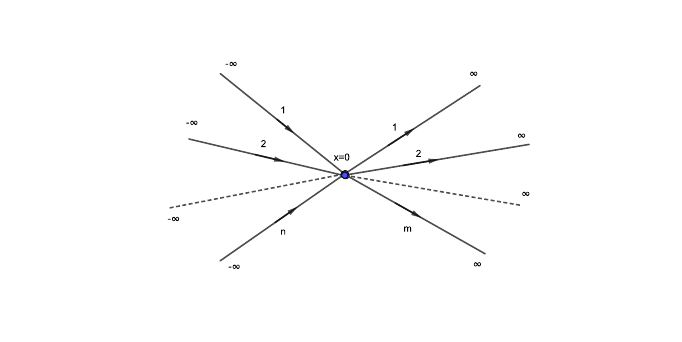}	
\caption{Star-shaped graph with $n$ incoming edges and $m$ outgoing edges at the junction point $x=0$.}
\label{fig1}
\end{center}
\end{figure}

On the structure $\Gamma=\cup_{k=1}^{n+m} I_k$  we consider a time evolution function  $\mathbf{u}:\rr_+ \times \Gamma \mapsto \mathbb{R}$
with $n+m$ components described by
\begin{equation}\label{car density}
u_i:\mathbb{R}_{+}\times I_i \rightarrow \rr\,\mbox{ and }\,\, u_j:\mathbb{R}_{+}\times I_j \rightarrow \rr,
\end{equation}
with $i\in\{1, \ldots, n\}$ and  $j\in \{n+1, \ldots, n+m\}$, which verify the following system of evolution equations in the edges plus coupling conditions at
$x=0$ and initial data
 \begin{equation}\label{main.eq}
	\left\{\begin{array}{ll}
		\partial_t u_i(t, x)+\partial_x(f(u_i(t, x))=\partial_{xx}u_i(t,x), \quad t>0,\ x<0,\  i\in \{1,\dots,n\}, \\[6pt]
		\partial_t u_j(t,x)+\partial_x(f(u_j(t,x))=\partial_{xx}u_j(t,x),\quad t>0,\ x>0,\  j\in \{n+1,\dots,n+m\},\\[6pt]
		u_i(t,0)=u_j(t,0), \qquad t>0, \ i,j\in\{1,\dots,n+m\},\\[6pt]
		\displaystyle \sum_{i=1}^n \big(f(u_i)-\partial_x u_i\big)(t,0)=
		 \sum _{j=n+1}^{n+m}\big(f(u_j)-\partial_x u_j\big)(t,0),\quad t>0,
		 \\[6pt]
			u_i(0,x)=u_{0i}(x), \qquad x<0,\  i\in \{1,\dots,n\},\\[6pt]
		u_j(0,x)=u_{0j}(x), \qquad x>0, \ j\in \{n+1,\dots,n+m\}.
	\end{array}\right.
\end{equation}

Let us remark that for smooth solutions vanishing at infinity the mass of the solution satisfies
\[
\frac{d}{dt}\Big(\sum_{i=1}^n \int_{-\infty}^0 u_i(t,x)dx +   \sum_{j=n+1}^{n+m} \int_{-\infty}^0 u_j(t,x)dx \Big)=
(n-m)f(0).
\]
So the system is conservative under the assumption
$$
(n-m)f(0)=0.
$$
This condition will appear several times in our analysis.

 We will denote the initial data as $\mathbf{u}_0=(u_{0k})_{k=1}^{n+m}$.
 Here, our analysis is devoted to two classical issues: first, to prove
  the global well-posedness of system \eqref{main.eq} for a general class of  functions $f\in C^1(\rr)$ and then
 to study the long time behavior of the solutions in the case $f(s)=-a|s|^{q-1}s$ with $q\geq 2$ and $a\in \rr$ such that $(n-m)a\geq 0$.
 Previous results regarding the well-posedness can be found in \cite{MR2679594} for nonlinearities
 $f\in C_c^2((0,1))$.  For a comparison of our results
 with the previous ones we refer to Section \ref{mainresults}.

{\bf Functional framework.}
In order to  rigorously describe the main results of the paper we need to establish the
functional framework in which we carry on our analysis.  As expected, Sobolev spaces plays a crucial role.
We follow the terminology in \cite[Section 1.3]{MR3013208}.
We introduce  $C(\Gamma)$ the space of continuous functions on $\Gamma$.
A function $\bu=(u_1,\dots, u_{m+n})$  is continuous on $\Gamma$ if and only if $u_k$ is continuous on $I_k$ and its components take the same value at  the common vertice, that is,
$$u_i(0)=u_j(0), \quad \forall \ i,j=1,\dots, m+n.$$
In a similar way we can define $C^l(\Gamma)$, their components are $C^l$ on each edge and their derivatives takes the same values at the common vertex. The functions $C_c^l(\Gamma)$ are the functions which vanishes outside a compact set of   graph  $\Gamma$.
For the  half lines $I_k$  described above, we consider the Hilbert space
$L^2(\Gamma)=\prod_{k=1}^{n+m} L^2(I_k)$
with the inner product
$$(\bu,\bv)_{L^2(\Gamma)}=\displaystyle\sum_{k=1}^{n+m}\int_{I_k} u_k v_k dx,$$
for $\bu=(u_1,\ldots,u_{m+n}), \bv=(v_1,\ldots,v_{m+n})\in L^2(\Gamma)$. Then, by $\widetilde H^l(\Gamma)$, we understand
the Hilbert space
$\widetilde H^l(\Gamma)=\bigoplus_{k=1}^{n+m} H^l(I_k)$
endowed with its natural norm. We also introduce the space  $H^1(\Gamma)=\widetilde H^1(\Gamma)\cap C(\Gamma)$:
\[
H^1(\Gamma)=\{\bu\in \widetilde H^1(\Gamma) ; \ u_i(0)=u_j(0), \ \forall\, 1\leq i,j\leq m+n\}.
\]
 This is a Hilbert space endowed with the inner product induced by the one on $H^1(\Gamma)$:
$$(\bu,\bv)_{H^1(\Gamma)}=\displaystyle\sum_{k=1}^{n+m}\int_{I_k} u_k v_k dx + \int_{I_k} \partial_x u_{k} \partial_x v_{k} dx.$$
The above definitions extends to any graph $\Gamma'$ with finitely many edges.  
In general for a graph $\Gamma$ we define $H_0^1(\Gamma)$ the closure of $C^1_c(\Gamma)$ in the $H^1(\Gamma)$-norm. Its dual $(H_0^1(\Gamma))'$ will be denoted by $H^{-1}(\Gamma )$.
These spaces have similar embedding properties to the classical real line case \cite[Section 3.2]{mugnolo}.

We introduce the Laplace operator $\dg$ on the network $\Gamma$ as follows:
$\dg:D(\dg) \subset L^2(\Gamma) \rightarrow L^2(\Gamma)$
given by
\begin{equation}\label{A}
\begin{cases}
\dg \bu = (\partial_{xx}u_{1},\dots, \partial_{xx}u_{m+n}), \,\,\mbox{ where } \bu=(u_1,\ldots,u_{m+n}),\\[5pt]
D(\dg) =\Big\{\bu\in \widetilde H^2(\Gamma) ; u_i(0)=u_j(0),\forall\, 1\leq i,j\leq m+n,\, \displaystyle\sum_{i=1}^{n}\partial_xu_{i}(0)=\sum_{j=n+1}^{n+m}\partial_xu_{j}(0)\Big\}.
\end{cases}
\end{equation}
It is easy to check that $\Delta_\Gamma$ is a self-adjoint operator.
The quadratic form associated to the operator $\Delta_\Gamma$ is
\[
\mathcal{E}(\bu,\bu)=\int_{\Gamma} |\partial_x \bu|^2 dx=\sum_{k=1}^{m+n}\int_{I_k}|\partial_x u_{k}|^2dx,
\]
with domain $D(\mathcal{E})=H^1(\Gamma)$.
Since $\mathcal{C}^{\infty}_c(I_k)$ is dense in $L^2(I_k)$ and  $\prod_{k=1}^{n+m}\mathcal{C}^{\infty}_c(I_k)\subset D(\dg)$
it follows that $D(\dg)$ is dense in $L^2(\Gamma)$. Moreover, $(\dg,D(
\dg))$ is a closed operator.

 Notice that we have used the notation $\partial_x u$ in the definitions of the operator $(\dg,D(
	\dg))$ and its associated quadratic form.   In principle, this is an abuse of notation because  it refers to the derivative of a function with one variable argument (which is usually denoted by $u^\prime$).
	Our notation becomes useful when used in the context of the evolution system \eqref{main.eq}
	where two different variables,  $t$ and $x$, appear in the argument of $u$.

\section{Main results}\label{mainresults}

Let us state our main results. To simplify the presentation we will denote $V=H^1(\Gamma)$ and  we also denote by $V'$ the dual of $V$.

First, we have the following general global well-posedness result.
\begin{theorem}[Global well-posedness]\label{well-posedness-L2} Let  $f\in C^1(\rr)$
satisfying
\begin{equation}
\label{h1}
(n-m)\limsup _{x\rightarrow -\infty}f(x)\geq 0\geq (n-m)\liminf _{x\rightarrow \infty}f(x).
\end{equation}
For any $\bu_0=(u_{0k})_{k=1}^{n+m}\in L^2(\Gamma)\cap L^\infty(\Gamma)$ 	there exists a unique weak solution $\bu=(u_k)_{k=1}^{n+m}$
	of problem \eqref{main.eq} (in the sense of \eqref{weak-solution}) which satisfies
	\[
	\bu\in C([0,\infty);L^2(\Gamma))\cap L^\infty((0,\infty)\times \Gamma)\cap L^2((0,\infty);V).
	\]
	Moreover, if $\um\leq u_{0k}(x)\leq \om$ for all  $k\in \{1, \ldots, n+m\}$, for some constants $\um  \leq 0\leq\om$, satisfying
\begin{equation}
\label{hM}
(n-m) f(\um)\geq 0\geq (n-m)f(\om),
\end{equation}	
then
	\[
	\um\leq u_k(t,x)\leq \om, \quad\forall t>0,\   x\in I_k, \  1\leq  k\leq   n+m.
	\]
\end{theorem}

When the initial datum is also integrable, as a consequence of Theorem \ref{well-posedness-L2}, we obtain the following result.

\begin{theorem}\label{well-posedness} Let  $f\in C^1(\rr)$ satisfying \eqref{h1}.
	For any $\bu_0\in L^1(\Gamma)\cap L^\infty(\Gamma)$ there exists a unique solution $\bu=(u_i)_{i=1}^{n+m}$
 of problem \eqref{weak-solution} such that $
\bu\in C([0,\infty);L^1(\Gamma))\cap L^\infty((0,\infty)\times \Gamma)\cap L^2((0,\infty);V)
 $.
Moreover,  the solution satisfies the following properties:

i) $L^1$-Stability.
Consider two different initial data $(u_{l0})_{l=1}^{n+m}$ and $(\tilde u_{l0})_{l=1}^{n+m}$ and
$\bu=(u_l)_{l=1}^{n+m}$ and $\widetilde{\bu}=(\tilde{u}_l)_{l=1}^{n+m}$ the corresponding solutions.
Then
 \begin{align}
 \nonumber	\sum _{i=1}^n \|u_{i}(t)-\tilde u_i(t)\|_{L^1(\rr_-)}&+\sum _{j=n+1}^{n+m} \|u_{j}(t)-\tilde u_j(t)\|_{L^1(\rr_+)}\\
 	\leq \sum _{i=1}^n& \|u_{i0}-\tilde u_{i0}\|_{L^1(\rr_-)}+\sum _{j=n+1}^{n+m} \|u_{j0}-\tilde u_{j0}\|_{L^1(\rr_+)}.
 \end{align}

 ii) Mass conservation. Assume that $(n-m)f(0)=0$. Then
		 \[
		 \int_\Gamma \bu(t,x)dx=\int _\Gamma \bu_{0}(x)dx, \quad \forall t\geq 0.
		 \]
 \end{theorem}

The precise notion of weak solutions for system \eqref{main.eq} will be introduced and analyzed in Subsection \ref{weak-sol-section}.

  Particular nonlinearities satisfying \eqref{h1} are   $f(s)=\pm a |s|^{q-1}s$  where $\pm (n-m)a\leq 0$. The nonlinearities $f(s)=|s|^p$ with $p>1$ are not covered for all initial data in $L^2(\Gamma)\cap L^\infty(\Gamma)$. However the same analysis works for nonnegative (nonpositive)
  solutions if $n-m\leq  0$ (respectively $n-m \geq 0$).

  Other type of nonlinearities  were considered in \cite[Th. 1.2]{MR2679594} where the authors studied the case of a bounded initial datum lying between $0$ and $1$ and a nonlinearity
  $f\in C_c^2((0,1))$.
  The results in \cite{MR2679594} still hold if one replaces $(0,1)$ by any finite interval $(a,b)$.
   Our Theorem \ref{well-posedness-L2} applies for more general situations.
In contrast with the analysis in \cite{MR2679594}, based on the semigroups approach, for the proof of Theorem \ref{well-posedness-L2} we use an
iterative method in the weak formulation \eqref{weak-solution} of the problem.

The semigroup method in \cite{MR2679594}
  uses an iterative strategy as in \cite[Appendix B]{MR3443431}.
An important step in the existence proof in \cite{MR2679594} is that the semigroup $S(t)$ generated by the m-dissipative operator $\Delta_\Gamma$ commutes with the derivative    $\partial_x$,  that is $  S(t) \partial_x f = \partial_x (S(t) f)$, (strategy also used for the Cauchy problem in   \cite{MR1124296,MR3443431}),   which is not true for the type of problems we consider here (see Appendix \ref{semigroup-approach} and \cite[Lemma 2.1]{MR3834707}). Hence, we cannot use the classical strategy valid in the whole space $\rr^N$ that is based on a fix point argument in the variation of constants formula to prove the local existence of solutions (as done in \cite{MR1124296}). On the other hand, a rigorous proof  using the semigroup approach of the existence of solutions for a nonlinear problem, the KdV equation, on a star-shaped network is given in \cite{MR3799048}.

 Once the well-possedness of our problem is settled, we focus our attention in the asymptotic behavior of the solutions.
 Next, we prove that for some particular nonlinearities $f$ the solutions to the system \eqref{main.eq} behave as some
	self-similar profiles
	when time goes to infinity. In this regard, our main results are the following. First, we prove some useful estimates.

\begin{theorem}\label{decay.estimates}
	Let $f(s)=-a |s|^{q-1}s$, with $q\geq 2$, and $a\in \rr$ such that  $(n- m)a\geq 0 $.
	For any initial datum $\bu_0\in L^1(\Gamma)\cap L^\infty(\Gamma)$
	there exists a unique global weak solution $
\bu\in C([0,\infty);L^1(\Gamma))\cap L^\infty((0,\infty)\times \Gamma)\cap L^2((0,\infty);V)
 $ of system \eqref{main.eq} in the sense of \eqref{weak-solution}. Moreover, the solution satisfies
  	\begin{enumerate}
	\item  Mass conservation,
		 \[
		 \int_\Gamma \bu(t,x)dx=\int _\Gamma \bu_{0}(x)dx, \quad \forall t\geq 0.
		 \]
		 \item Energy estimate: for any $0\leq t_1<t_2<\infty$ we have
		 \[
		 \int_\Gamma \bu^2(t_2,x)dx+ 2 \int _{t_1}^{t_2}\int_\Gamma |\partial_x\bu|^2(t,x)dx
		 \leq  \int_\Gamma \bu^2(t_1,x)dx.
		 \]
		 \item Decay estimate: for any $1\leq p\leq \infty$ it holds that
		 \[
		 \|\bu(t)\|_{L^p(\Gamma)}\leq C(p) \|\bu_0\|_{L^1(\Gamma)}t^{-\frac 12(1-\frac 1p)}\quad \forall t> 0.
		 \]
	\end{enumerate}
\end{theorem}

\begin{remark}\label{l1solutions} 
When the initial data is only in $L^1(\Gamma)$ a more detailed analysis as in \cite{MR1124296}, \cite[Th.~4.7, p.~35]{zuazua2020asymptotic} shows the existence of a unique solution $\bu\in C([0,\infty),L^1(\Gamma))\cap L^\infty_{loc}((0,\infty),L^\infty(\Gamma))$. It is a consequence of the $L^1(\Gamma)$-contraction property and an approximation argument. Indeed, choose $u_{0n}\in L^1(\Gamma)\cap L^\infty(\Gamma)$ such that $u_{0n}\rightarrow u_0$ in $L^1(\Gamma)$. The corresponding solutions satisfy  $
\bu_n\in C([0,\infty);L^1(\Gamma))\cap L^\infty((0,\infty)\times \Gamma)\cap L^2((0,\infty);V).
 $ 
The $L^1(\Gamma)$-contraction property shows that $(u_n)_{n\geq 1}$ is a Cauchy sequence in $C([0,\infty);L^1(\Gamma))$. Its limit $\bu$ has the same   decay properties as  $\bu_n$, in particular the estimate in $L^\infty(\Gamma)$.
 Moreover, for any $\tau>0$ the solution $\bu$ belongs to $L^2((\tau,\infty);V)$ and $\partial_t\bu\in L^2((\tau,\infty);V')$. It satisfies the weak formulation in  \eqref{weak-solution} and takes $u_0$ as initial datum. For the interested reader the full details are the same as in  \cite[Th.~4.7, p.~35]{zuazua2020asymptotic}.
 \end{remark}

\begin{theorem}[Asymptotic behavior for large times]\label{first.term}  Let $f(s)=-a |s|^{q-1}s$, with $q\geq 2$, and $a\in \rr$ such that  $(n- m)a\geq 0 $.
	For any initial datum $\bu_0\in L^1(\Gamma)\cap L^\infty(\Gamma)$
	denoting by $M$ the total mass of the initial  datum, i.e.,
	$$M=\int_\Gamma \mathbf{u}_0(x)dx:=\sum_{i=1}^{n} \int_{-\infty}^{0} u_{0i}(x) dx+  \sum_{j=n+1}^{n+m} \int_{0}^{\infty} u_{0j}(x)dx,$$
	we have that the solution of system \eqref{main.eq} satisfies:
	\begin{enumerate}
	\item[i)] if $q>2$
	\begin{equation}
	\label{q>2}
	t^{\frac 12(1-\frac 1p)}\| \bu(t)-\bu_M(t)\|_{L^p(\Gamma)}\rightarrow 0, \ \text{as}\ t\rightarrow \infty, \ 1\leq p<\infty,
	\end{equation}
	where $\bu_M(t)=(u_{M,k}(t))_{k=1}^{m+n}$ is given by
	\[
	u_{M,k}(t,x)=\frac{2M}{m+n}\frac {1}{\sqrt {4\pi t}}e^{-\frac{x^2}{4t}},
	\]
	\item[ii)] if $q=2$ and $\bu_0$ is nonnegative (or nonpositive) then
	\begin{equation}
	\label{q=2}
	t^{\frac 12(1-\frac 1p)}\| \bu(t)-\bu_M(t)\|_{L^p(\Gamma)}\rightarrow 0, \ \text{as}\ t\rightarrow \infty, \ 1\leq p<\infty,
	\end{equation}
	where   $\bu_M(t)=(u_{M,k}(t))_{k=1}^{m+n}$ is given by
	\[
	u_{M,k}(t,x)=\frac 1{a\sqrt t}G_{aM}(\frac x{\sqrt {t}})\]
	 where for a given parameter $M$
	\begin{equation}
\label{perfil.fm.intro}
  G_M(y)= \frac{\alpha_{n,m,M} e^{-\frac{y^2}4}}{1+\alpha_{n,m,M}\int_{-\infty}^y e^{-s^2/4}ds}
\end{equation}
	and $\alpha_{n,m,M} $ is the constant obtained as the unique solution in the interval $ (-\frac 1{2\sqrt \pi}, \infty)$
	of the equation
	\[
	 |1+\alpha \sqrt{ \pi} |^{n-m}|1+2\alpha \sqrt\pi|^m=e^M.
	\]
	\end{enumerate}
\end{theorem}
%

 Now, let us   comment on the hypotheses of Theorem \ref{first.term}.
The condition   $(n-m)a\geq 0$ in Theorem \ref{decay.estimates} and  Theorem  \ref{first.term} is imposed in order to guarantee the global existence of the solutions. However, we do not
know if this assumption is merely technical or  if solutions blow-up when  $(n-m)a< 0$.

 The case $q=2$ when the initial datum changes sign
remains to be treated elsewhere. To deal with this case will be necessarily to prove the uniqueness of the solutions of the system \eqref{main.eq} with initial  datum  $\delta_0$ taken in the sense of bounded measures (see Section \ref{self-similar} for a precise definition).
Similar difficulties appeared previously in the case of the whole space where the uniqueness of the profiles was addressed in a series of papers.  We quote \cite{MR1233647} for nonnegative/nonpositive solutions on the real line, \cite{MR1266100} for multi-dimensional case and \cite{MR1440033} for changing sign solutions.

The decay of solutions of linear diffusion problems on graphs has been analyzed previously in \cite{MR2291812} for compact graphs and in \cite{haeseler} in the case when some infinite edges are attached to the compact part of the graph. The asymptotic expansion of the solutions for these  linear problems has been done for general graphs in \cite{rossi2020}.
We expect that the results of this paper can be extended to   connected finite graphs in which some of the edges have infinity length. In this case, depending on the nonlinearity, additional restrictions on the number of incoming and outgoing edges at any internal node have to be imposed, for example that always the number of incoming edges is greater or equal than the number of outgoing edges. This will be analyzed in the future.

Let us now comment on the ideas and methods used in the proofs of Theorems \ref{well-posedness-L2}-\ref{first.term}.
The ideas involved in the proof of Theorem \ref{well-posedness-L2} can be sketched as follows: we will first obtain a global well-possedness result when $\bu_0\in L^2(\Gamma)$ and the nonlinearity $f$ is Lipschitz. Then, combining this result with a priori estimates and some classical arguments in conservation laws (see, for instance, \cite[p.~60]{MR1304494}) we can extend the global well-possedness to the case of  initial data in $L^2(\Gamma)\cap L^\infty(\Gamma)$ and a $C^1(\rr)$ nonlinearity.
Using apriori estimates we establish well-posedness results for $L^1(\Gamma)\cap L^\infty(\Gamma)$ initial data.
In the particular case when $f(s)=-a |s|^{q-1}s$ by energy estimates we obtain qualitative properties of solutions that allow us to deal with $L^1(\Gamma)$ solutions and to establish in Theorem \ref{first.term} the long time behavior of the solutions.
We will prove Theorem \ref{first.term} by using a scaling method, i.e., we
introduce a family of scaled solutions $\{u^\lambda\}_{\lambda > 1}$ and reduce the asymptotic expansion property \eqref{q>2}
to the strong convergence of the scaled family $\{u^\lambda\}_{\lambda > 1}$. We will make use of
energy-type estimates which provide uniform bounds, with respect to $\lambda$, of the scaled solutions and allow
us to pass the limit by using the Aubin-Lions compactness criterium. We refer to \cite{MR2656972,Karch2011} for a review of the
scaling method. The main difficulty in the case $q=2$ is to prove the uniqueness of the solutions of the limit equation.

The paper is organized as follows: Section \ref{well-posed} deals with technical a priori estimates for the solutions and the global well-possedness result.
In Section \ref{self-similar} we discuss the existence and uniqueness of the self similar profiles that are used to characterize the long time behavior of the solutions. Section \ref{asymptotic behavior} is devoted to prove our main result stated in Theorem \ref{first.term}. Finally, in Appendix \ref{semigroup-approach} we explain the difficulties in using the semigroup approach for the problem addressed here.

\section{Well-possedness}\label{well-posed}

\subsection{Weak solutions}\label{weak-sol-section}
First we introduce  the notion of weak solution to our system \eqref{main.eq}.
Let us assume that we have a classical solution $\bu=(u_1,\ldots,u_{n+m})$ of problem \eqref{main.eq} on the time interval $[0,T]$, i.e. all the components $u_k$, $k=1,\dots, n$, are $C^{1,2}([0,T]\times I_k)$  and   together with their derivatives decay to zero at infinity.
Take $\bpsi=(\psi_1,\ldots,\psi_{n+m}) \in C_c^2(\Gamma)$, i.e. a  function with all the components of class $C^2$ on the edge where they are  defined and vanishing outside of a compact set. 
We multiply the equation in \eqref{main.eq} by $\bpsi$ and integrate, we obtain that
\begin{align*}
 (\partial_t\bu, \bpsi)_{L^2(\Gamma)}
& =\sum_{i=1}^n \int_{-\infty}^0 (\partial_{xx}u_i - \partial_x (f(u_i))\psi_i dx
+\sum_{j=n+1}^{n+m} \int_0^\infty (\partial_{xx}u_j - \partial_x (f(u_j))\psi_j dx\\
&=\sum_{i=1}^n (\partial_x u_i - f(u_i))(0)\psi_i(0)-\sum_{j=n+1}^{n+m} (\partial_x u_j - f(u_j))(0)\psi_j(0)\\
&\qquad
 - (\partial_x \bu, \partial_x\bpsi)_{L^2(\Gamma)} + (f(\bu),\partial_x\bpsi)_{L^2(\Gamma)},
\end{align*}
where $f(\bu)$ is the vector defined by $f(\bu)=(f(u_1), \ldots, f(u_{n+m}))$.
Using  that $\bpsi=(\psi_1,\dots, \psi_{n+m})$ is continuous at $x=0$, i.e.  $\psi_i(0)=\psi_j(0):=\bpsi(\b0)$ for all $1\leq i,j\leq n+m$,  we obtain
 \[
 \left(\partial_t\bu,\bpsi\right)_{L^2(\Gamma)} + (\partial_x \bu, \partial_x\bpsi)_{L^2(\Gamma)} = (f(\bu),\partial_x\bpsi)_{L^2(\Gamma)}.
 \]
By a density argument the above representation  holds for all $\bpsi\in V$. This means that a smooth solution $\bu$ is in fact a ``{\it weak solution}'' in the following sense:
\begin{definition}[weak solutions]\label{def1}
Let $T>0$. We say that a function $\bu= (u_{k})_{k=1}^{n+m}$ with the components $u_k :(0,  T)\times I_k  \rightarrow \mathbb{R}$ is a weak solution of \eqref{main.eq}
with $\bu_0\in L^2(\Gamma)$.  if
\begin{equation}
\label{weak-solution}	
\left\{
 \begin{aligned}
& \bu\in L^2(0,T; V),  \partial_t\bu\in L^2(0,T; V'),  \\
&  \mbox{for a.e.}\  t \  \text{in} \, (0,T), \text{the following holds for all } \bpsi\in V\\
&\left\langle\partial_t\bu(t, \cdot),\bpsi\right\rangle_{V',V} + (\partial_x \bu(t, \cdot), \partial_x\bpsi)_{L^2(\Gamma)} = (f(\bu(t, \cdot)) ,\partial_x\bpsi)_{L^2(\Gamma)},\\
&\bu(0, \cdot)=\bu_0(\cdot).
\end{aligned}
\right.
 \end{equation}
If \eqref{weak-solution} holds  for any $T>0$ we say that the solution is global.
 \end{definition}

 Since  $V$ is dense and continuous embedded in $L^2(\Gamma)$ following classical arguments (see, e.g. \cite[Lemma 1.2, p. 176]{temam}) the conditions  $\bu\in L^2(0,T; V), \, \partial_t\bu\in L^2(0,T; V')$ given above imply that $\bu\in {C}([0,T]; L^2(\Gamma))$. Moreover,
it holds that
 \[
\|\bu(t')\|_{L^2(\Gamma)}^2-\|\bu(t)\|_{L^2(\Gamma)}^2=2\int_t^{t'}\langle \partial_t \bu(s ),\bu(s )\rangle_{V',V} ds, \quad \forall \ 0<t<t'<T,
\]  
and
\[
\frac{d}{dt}\|\bu(t)\|_{L^2(\Gamma)}^2=2\langle \partial_t \bu(t ),\bu(t )\rangle_{V',V}, \quad \textrm{ for a.e. } t\in (0,T).
\]

\subsection{A priori estimates} In order to obtain the well-possedness  for our system we need some a priori estimates.
 First, we prove the following lemma.

\begin{lemma}\label{ro}
Assume that $f:\rr\rightarrow\rr$ is a  globally  Lipschitz function and $\bu\in L^2(0,T;V)$, $\partial_t\bu\in  L^2(0,T;V')$  is a weak solution   in the sense of   \eqref{weak-solution}. For any  function   $\rho\in W^{2,\infty}(\rr)$ with   $\rho(0)=\rho'(0)=0$
the following holds
\begin{align*}
\frac{d}{dt}\int_\Gamma \rho(\bu(t)) dx + \int_\Gamma (\partial_x \bu(t))^2 \rho''(\bu(t))dx&=
\int_{\Gamma} f(\bu(t))\rho''(\bu(t))  \partial_x\bu(t)dx \\
&=(n-m)\int _0^{\bu(t,0)}f(s)\rho''(s)ds,
\ \text{for a.e.}\ t\in (0,T),
\end{align*}
where by $\rho(\bu)$ we understand $\rho(\bu)=(\rho(u_1), \ldots, \rho(u_{n+m}))$.

Moreover, any two solutions $\bu$ and $\bv$ of \eqref{weak-solution} verify for a.e. $ t\in (0,T)$:
\begin{align*}
\frac{d}{dt}\int_\Gamma \rho(\bu(t)-\bv(t)) dx&+  \int_\Gamma (\partial_x \bu(t)-\partial_x \bv(t))^2 \rho''(\bu(t)-\bv(t))dx \\
& =
\int_{\Gamma} (f(\bu(t))-(f(\bv(t)) )\rho''(\bu(t)-\bv(t))\partial_x(\bu(t)-\bv(t))  dx.
\end{align*}
\end{lemma}

\begin{proof}
First we observe that $\rho'(\bu)\in L^2(0,T;V)$.
   Indeed, $\rho'$ is continuous, so    $\rho'(\bu(t))\in C( \Gamma)$ for a.e. $t\in (0,T)$ and it is sufficient to show that $\rho'(\bu)\in L^2(0,T,\widetilde H^1(\Gamma))$.
 Using the estimate  $|\rho'(s)|=|\rho'(s)-\rho'(0)|\leq \|\rho''\|_{L^\infty(\rr)}|s|$ we get $\rho'(\bu)\in L^2(0,T;L^2(\Gamma))$.  Also $|\partial_x [\rho^\prime(u)]|= |\rho''(\bu)\partial_x \bu |\leq  \|\rho''\|_{L^\infty(\rr)}|\partial_x \bu|\in L^2(0,T; L^2(\Gamma))$ which proves the desired property for $\rho'(\bu)$. Since we have $\rho(0)=\rho'(0)=0$ we get that
$|\rho(s)|\leq |s|^2 \|\rho''\|_{L^\infty(\rr)} $ and then $\rho(\bu(t))\in L^1(\Gamma)$ for a.e. $t\in [0,T]$.

Using that $\rho'(\bu)\in L^2(0,T;V)$ and $\partial_t\bu\in  L^2(0,T;V')$ we obtain that $\rho (\bu)\in C([0,T],L^1(\Gamma))$ and
it holds that
\[
\int_\Gamma \rho(\bu(s))dx-\int_\Gamma \rho(\bu(s'))dx=
\int_s^{s'} \langle \partial_t \bu(t),{ \rho^\prime(\bu(t))}\rangle_{V',V} dt, \quad \forall \ 0<s<s'<T.
\]
  Thus, from Definition \ref{def1}     choosing $\bpsi=\rho^\prime(\bu )$    we obtain for a.e. $t>0$
\begin{align*}
\displaystyle\frac{d}{dt}\int_\Gamma\rho(\bu )dx&=\langle \partial_t \bu , \rho'(\bu )\rangle_{V',V}  =
(-\partial_x  \bu  +f(\bu)  , \partial_x(\rho'(\bu )))_{L^2(\Gamma)} \\
&=-\int_\Gamma (\partial_x \bu  )^2 { \rho''(\bu )} dx + \int_\Gamma  f(\bu)  \rho''(\bu ) \partial_x \bu   dx.\end{align*}
Let us denote
$$H(r)=\int_0^r f(s)\rho''(s) ds.$$
Using that $f$ if globally Lipschitz  and $\rho''$ is bounded we obtain that $H'$ is bounded on bounded sets and $|H(s)|\leq A|s|+B|s|^2$. Since $u_k (t)\in H^1(I_k)$ we use that $H'$ is bounded on the set  $[-\|u_k(t)\|_{L^\infty(I_k)},\|u_k(t)\|_{L^\infty(I_k)} ]$ to obtain that   $H(u_k)\in H^1(I_k)$ and then for a.e. $t\in (0,T)$:
\[
 \int_\Gamma  f(\bu(t))  \rho''(\bu(t) ) \partial_x \bu  (t) dx=\sum _{i=1}^nH(u_i(t,x))\Big|_{x=-\infty}^{x=0}+\sum _{j=n+1}^{n+m}H(u_j(t,x))\Big|_{x=0}^{x=\infty}= (n-m)H(\bu(t,0)).
\]
 The estimate on $H$ implies in particular that $H(\bu(t,0))\in L^1(0,T)$ so the right hand side is well defined.


  Let us now take two solutions $\bu$ and $\bv$ of \eqref{weak-solution}. Thus,    choosing $\bpsi=\rho^\prime(\bu-\bv)$  in the weak formulation of $\bu$ and $\bv$  we obtain for a.e. $t>0$
\begin{align*}
\displaystyle\frac{d}{dt}\int_\Gamma\rho(\bu-\bv)dx&=\langle \partial_t(\bu-\bv), \rho'(\bu-\bv)\rangle_{V',V} \\
& =
(-\partial_x (\bu-\bv)+f(\bu)-f(\bv), \partial_x(\rho'(\bu-\bv)))_{L^2(\Gamma)}\\
&=(-\partial_x (\bu-\bv)+f(\bu)-f(\bv),  \rho''(\bu-\bv)\partial_x(\bu-\bv))_{L^2(\Gamma)}\\
&=-\int_\Gamma (\partial_x(\bu-\bv) )^2 { \rho''(\bu-\bv)} dx + \int_\Gamma (f(\bu)-f(\bv))\rho''(\bu-\bv) \partial_x(\bu-\bv) dx.
\end{align*}
This finishes the proof.
\end{proof}

For any smooth convex function $\rho$ we have that $\rho'' \geq 0$. Using smooth convex approximations of  convex functions we will obtain various properties of the solutions.
As a consequence of Lemma \ref{ro}, we obtain the following maximum principles and $L^p$ estimates on the solutions.

\begin{remark} {\rm
	In what follows and throughout the paper we will refer to the inequality $\bu_{0}\leq \bv_{0}$
	meaning that the inequality holds for each of the components (similarly for $\geq$ and for the
	comparison of two solutions $\bu \leq \bv$ or $\bu \geq \bv$).}
\end{remark}

\begin{corollary} \label{particular.apriori} {Let  $f:\rr\rightarrow\rr$ is a  globally  Lipschitz function. 
Assume  that   $\bu$ satisfies the weak formulation \eqref{weak-solution}.}We have

i) If $\bu_{0} \geq \um$ for   some constant $\um\leq 0 $ satisfying $(n-m)f(\um)\geq 0$, then  $\bu \geq \um$.

ii) If $\bu_{0} \leq \om$ for some constant $\om\geq 0$ satisfying $(n-m)f(\om)\leq 0$ then  $\bu \leq \om$.

iii) If  $\bu_0\leq \bv_0$     the corresponding weak solutions   satisfy   the comparison principle:
 $$\bu(t )\leq
\bv(t ), \quad \forall t\in (0,T).$$

iv) Assuming that $\bu\in C([0,T],L^p(\Gamma))$ with $1\leq p<\infty$, the evolution of the $L^p(\Gamma)$ norm of the solutions to \eqref{weak-solution} satisfies
\begin{equation}
\label{norm1}
\frac{d}{dt}\int_\Gamma |\bu(t,x)|  dx\leq  (n-m) f(0)\sgn(\bu(t,0)), \qquad p=1
\end{equation}
  and
\begin{equation}
\label{normp}
\frac{d}{dt}\int_\Gamma |\bu(t,x)|^p dx\leq p(p-1)(n-m)  \int _0^{\bu(t,0)}f(s) |s|^{p-2}ds, \qquad p>1.
\end{equation}
\end{corollary}

\begin{remark} {\rm
 System \eqref{main.eq} admits constant solutions ${ v\equiv C}$ if and only if $(n-m)f(C)=0$.
Under the assumption $(n-m)f(0)=0$ the solutions remain nonnegative (or nonpositive) if the initial data
are nonnegative  (or nonpositive). Also the $L^1$-norm does not increase.}
 \end{remark}
 
 {
 \begin{remark} When $f\in C^1(\rr)$ and the solutions are bounded we entry in the framework of globally Lipschitz nonlinearities.  Indeed, since $u \in L^\infty((0,T)\times \Gamma)$ we denote by $A=\|\bu\|_{L^\infty((0,T)\times \Gamma)}$ and consider a function $\tilde f\in C^1(\rr)$ such that $\tilde f(x)=f(x)$ for $x\in (-A,A)$ and $\tilde f$ vanishes outside the interval $(-A-1,A+1)$. It follows that $\bu$ is again a weak solution with $f$ replaced by $\tilde f$. Moreover the properties of function $f$ that appear in the Corollary remain unchanged.
 \end{remark}
}

\begin{remark} {\rm
The case considered in \cite{MR2679594} corresponds to $f\in C_c^2((0,1))$
with $0\leq u_0\leq 1$. In this case maximum principles in $i)$ and $ii)$ as well as the stability of the $L^1$-norm hold. The stability in the $L^p$-norms with $1<p<\infty$ holds under the assumption $n\leq m$. }
\end{remark}

\begin{proof}[Proof of Corollary \ref{particular.apriori}]
{
We will make use of Lemma \ref{ro} for particular functions $\rho$.
}
 As in  \cite[p.185, Proof of Th. 6.3.2]{MR3468916} let us first consider the function $\eta_\eps\in C^1(\rr)$
  given by
  \[
\eta_{\eps}(s)=
\begin{cases}
	0, & -\infty<s\leq 0,\\
	\displaystyle \frac{s^2}{4\eps}, & 0<s\leq 2\eps,\\
	s-\eps, & 2\eps<s<\infty.
\end{cases}
\]
It is immediate that $0\leq \eta_\eps(s)\leq s^+$  and $\eta_\eps(s)\rightarrow s^+$       as $\eps\rightarrow 0$.

{\it i)} For the first part we consider the convex approximation $\rho_\eps(t)=\eta_\eps(\um-t) $ of the function   $(\um-t)^+$. Notice that since $\eta_\eps(s)\rightarrow s^+$ as $\eps\rightarrow 0$, it holds that
$\rho_\eps(t)=\eta_\eps(\um-t) \rightarrow (\um-t)^+$ as $\eps\rightarrow 0$ and, moreover, 
$\rho''_\eps(s)$ is concentrated in $[\um-2\eps,\um]$.
We have that $\rho_\eps(\bu(0))\equiv 0$, $\rho_\eps''={\bf 1}_{(\um-2\eps,\um)}/2\eps$
and 
\[
 H_\eps(t):= 
\int_0 ^{t}f(s)\rho''_\eps(s)ds=- 
\int^0 _{t}f(s)\rho''_\eps(s)ds.
\]
Since $\um\leq 0$   the support of $\rho''_\eps$ guarantees that the above integral vanishes when $t\geq 0$. The same holds when $\um\leq t\leq 0$. When $ t<\um$ we have for $\eps$ small enough that
\begin{align*}
 H_\eps(t)& =-\frac{1}{2\eps} \int_{\um-2\eps}^{\um}f(s)ds\rightarrow -f(\um), \quad\eps\rightarrow 0 .
\end{align*}
Hence, $H_\eps(t)\rightarrow -f(\um)\sgn(\um-t)^+$ and 
\[
\lim_{\eps \rightarrow 0}(n-m)H_\eps(\bu(t,0))=-(n-m)f(\um)\sgn (\um-\bu(t,0))^+
\leq 0.\]
Moreover, the support of $\rho''$ gives us that for small $\eps$ we have $|H_\eps(\bu(t,0))|\leq \|f\|_{L^\infty((\um-1,\um))}$. Letting $\eps\rightarrow 0$   by Fatou's Lemma we obtain that $(\um-\bu(t))^+\in L^1(\Gamma)$ and it vanishes identically. This proves the first case.

{\it ii)} For the second part we consider a convex approximation $\rho_\eps(t)=\eta_\eps(t-\overline M)$ of the function $ (t-\om)^+$.
It follows that $\rho_\eps''={\bf 1}_{(\om, \om+2\eps)}/2\eps$ and using that $\om\geq 0$ we get
\begin{align*}
H_\eps(t)&:=\int_0 ^{t}f(s)\rho''_\eps(s)ds 
 \rightarrow
\begin{cases}
f(\om),& t>\om,\\
	0, & t\leq \om.
\end{cases}
\end{align*}
Similar arguments as in the previous case give us the desired property.

{\it iii)} Let us choose $\psi_R(x)=\psi(x/R)$ where $\psi\in C_c^\infty(\rr)$, $0\leq \psi\leq 1$, satisfying $\psi\equiv 1$ in $(-1,1)$ and $\psi\equiv 0 $ in $|x|>2$. Let us consider the function $\bpsi_R=(\psi_k)_{k=1}^{m+n}$ where $\psi_{k}=\psi_R$. We will use as test function in the weak formulation of $\bu$ and $\bv$ function $\bpsi=\eta'_\eps(\bu-\bv) \bpsi_R$.
 It follows that \begin{align*}
\int_\Gamma &\eta_\eps(\bu(t)-\bv(t))\bpsi_R dx  -\int_\Gamma \eta_\eps(\bu_0-\bv_0) \bpsi_R dx\\
&=\int_0^t (-\partial_x(\bu-\bv)+f(\bu)-f(\bv), \partial_x(\eta_\eps'(\bu-\bv)\bpsi_R))ds\\
&=\int_0^t (-\partial_x(\bu-\bv)+f(\bu)-f(\bv), \eta_\eps'(\bu-\bv)\partial_x\bpsi_R)ds\\
&\quad - \int_0^t (\partial_x(\bu-\bv) , \eta_\eps''(\bu-\bv)\partial_x(\bu-\bv)\bpsi_R)ds+\int_0^t ( f(\bu)-f(\bv), \eta_\eps''(\bu-\bv)\partial_x(\bu-\bv)\bpsi_R)ds\\
&=I_\eps+II_\eps+III_\eps.
\end{align*}
Since $f$ is Lipschitz and $|\eta_\eps'|\leq 1$ the first term is bounded by
\[
|I_\eps|\lesssim \|u-v\|_{L^2(0,T;V))}\|\partial_x\bpsi_R\|_{L^2(0,T;L^2(\Gamma)))}\lesssim R^{-1/2}.
\]
The second one satisfies $II_\eps\leq 0$. For the third one since $s\eta_\eps''(s)= \frac{s}{2\eps}{\bf 1}_{(0, 2\eps)} (s)\leq 1$ we have
 $$|(f(\bu(s))-f(\bv(s)))  \eta_\eps^{\prime\prime}(\bu(s)-\bv(s))|\leq L  |(\bu(s)-\bv(s))\eta_\eps^{\prime\prime}(\bu(s)-\bv(s))|\leq L, \mbox{ for } s\in (0, t).$$
Using that $\partial_x(\bu-\bv)\bpsi_R\in L^1((0,T)\times \Gamma)$ and that $s\eta_\eps''(s)\rightarrow 0$ as $\eps\rightarrow 0$ we obtain that $III_\epsilon$ goes to zero as 
 as $\eps\rightarrow 0$. Letting $\eps\rightarrow 0$ we obtain that 
 \[
 \int_\Gamma (\bu(t)-\bv(t))^+ \bpsi_R(x)dxds\lesssim R^{-1/2}.
 \]
 Letting $R\rightarrow \infty$ we obtain that $(\bu(t)-\bv(t))^+\equiv 0$ and we finish the proof.

{\it iv)} We consider $\theta_\eps(t)=\eta_\eps(t)+\eta_\eps(-t)$ which approximates the function $|t|$ and
$\rho_\eps(s)=\theta_\eps^p(s)$. It follows that
\begin{equation}\label{deriv.norma}
 \displaystyle\frac{d}{dt}\int_\Gamma \rho_\eps(\bu(t)) dx \leq (n-m)\int _0^{\bu(t,0)}f(s)\Big[p(p-1)\theta_\eps^{p-2}(s)(\theta_\eps'(s))^2+p
 \theta_\eps^{p-1}(s)\theta_\eps''(s) \Big]ds.	
\end{equation}
When $p=1$ we have to analyze $H(t)=\int_0^t f(s)\theta_\eps(s)ds.$ The explicit 
representation of $\theta_\eps ''$ give us that
\[
H(t)=\int_0^t f(s)\frac{{\bf 1}_{(-2\eps,2\eps)}}{2\eps} ds \rightarrow f(0)\sgn(t), \quad \eps\rightarrow 0.
\] 
Thus we obtain that the right hand side in \eqref{deriv.norma} converges to $(n-m)f(0)\sgn(\bu(t,0)) $.

For $p>1$ we use that $\theta''_\eps={\bf 1}_{(-2\eps,2\eps)}/2\eps$ and $\theta_\eps(s)=s^2/4\eps$ in the support of $\theta''_\eps$. Hence the last term in the right hand side of \eqref{deriv.norma} goes to zero when $\eps\rightarrow 0$.  Concerning
the first term: since $0\leq \theta_\eps(s)\leq |s|$ and the map $s\mapsto f(s)|s|^{p-2}$ is locally integrable, using again the dominated convergence theorem, we can pass to the limit as $\eps \rightarrow 0$ to obtain   that
\[
\frac{d}{dt}\int_\Gamma |\bu(t,x)|^p dx\leq p(p-1)(n-m)   \int _0^{\bu(t,0)}f(s) |s|^{p-2}ds,
\]
which finishes the proof.
\end{proof}

\subsection{Existence and uniqueness of weak solutions}
In this subsection we show the well-possedness of our system \eqref{main.eq}, i.e., we prove Theorem \ref{well-posedness-L2}.
Let us explain the main steps in the proof of Theorem \ref{well-posedness-L2}.  First, we treat the case when $\bu_0\in L^2(\Gamma)$ and the nonlinearity $f$ is Lipschitz. Next, we consider the initial datum in $L^2(\Gamma)\cap L^\infty(\Gamma)$ and the nonlinearity $f$ in $C^1(\rr)$.
When we assume that the initial  datum is also integrable we obtain the contraction property in $L^1(\Gamma)$.

In order to obtain the existence of the solutions in each of the cases described above we remind the
following abstract theorem due to J. L. Lions (see \cite[Theorem X.9]{MR2759829}):
\begin{proposition}\label{lions}
Given $F\in L^2(0,T; V')$ and $\bu_0\in L^2(\Gamma)$, there exists a unique function
$\bu\in L^2(0,T; V)\cap {C}([0,T]; L^2(\Gamma)$, $\partial_t\bu \in L^2(0,T; V')$, such that
\[
\left\{
 \begin{aligned}
&\left\langle\partial_t\bu(t),\bpsi\right\rangle_{V',V} + (\partial_x \bu(t), \partial_x\bpsi)_{L^2(\Gamma)}
= \displaystyle\langle F(t),\bpsi\rangle_{V',V},\,\,\mbox{a.e. on } \, [0,T],\ \forall  \bpsi\in V,\\[6pt]
&\bu(0)=\bu_0.
\end{aligned}
\right.
\]
\end{proposition}

\begin{proof}[Proof of Theorem \ref{well-posedness-L2}.]
\textbf{Step I. Existence of solutions for Lipschitz  nonlinearities $f$ and $\bu_0\in L^2(\Gamma)$.}
We construct a sequence of functions inductively by considering solutions to the following problems:
\begin{equation}\label{aprox}
	\begin{cases}
\langle\displaystyle\partial_t\bu^{k+1}(t),\bpsi\rangle_{V',V} + (\partial_x \bu^{k+1}(t), \partial_x\bpsi)_{L^2(\Gamma)}
=  ( f(\bu^k(t)),\partial_x\bpsi)_{L^2(\Gamma)},\mbox{a.e.} [0,T],\,\forall\, \bpsi\in V,\\[6pt]
\bu^0\equiv 0,\quad \bu^{k+1}(0)=\bu_0.
		\end{cases}
\end{equation}
Let us first observe that
 since the  $f$ is Lipschitz continuous
	and that $\bu^k\in L^2(0,T; V)$ implies that $f(\bu^k)-f(0)\in L^2(0,T; V)$ we get that the map
	$F_k : (0,T) \to V'$ defined by
	\[
	\langle F_k(t),\bpsi \rangle_{V',V}=( f(\bu^k(t)),\partial_x\bpsi)_{L^2(\Gamma)}=( f(\bu^k(t))-f(0),\partial_x\bpsi)_{L^2(\Gamma)}+
	(n-m)f(0)\bpsi(\b0)
	\]
belongs to $L^2(0,T,V')$.
Proposition \ref{lions} gives us a sequence of functions $\{\bu^n\}_{n\in \mathbb{N}} \in W(0,T),$
where, for any $T>0$,
\begin{align}\label{W}
W(0,T)=\{\varphi \in  L^2(0,T; V); \partial_t\varphi\in L^2(0,T; V')\}\hookrightarrow   {C}([0,T]; L^2(\Gamma)).
\end{align}
 The next steps are devoted to prove that $\{\bu^n\}_{n\in \mathbb{N}}$ is a Cauchy sequence in $W(0,T)$ for a time $T$ depending on the Lipschitz constant of the nonlinearity $f$.  Therefore, we define
$$ \beta^{n}:=\bu^{n+1} - \bu^n,  n\geq 0,   $$
which belongs to $W(0,T)$ and, according to \eqref{aprox}, satisfies
 \begin{equation}\label{aprox1}
	\begin{cases}
\langle\displaystyle\partial_t\beta^n(t),\bpsi\rangle_{V',V} + (\partial_x \beta^n(t), \partial_x\bpsi)_{L^2(\Gamma)}
= \displaystyle\left( f(\bu^{n}(t))- f(\bu^{n-1}(t)),\partial_x\bpsi\right)_{L^2(\Gamma)},
\,\forall\, \bpsi\in V,\\[6pt]
\beta^n(0)=0.
		\end{cases}
\end{equation}
Choosing $\bpsi=\beta^n$ in \eqref{aprox1}, we find that
$$\displaystyle\frac{1}{2}\frac{d}{dt}\int_\Gamma |\beta^n|^2 dx +\int_\Gamma |\partial_x\beta^n |^2 dx
=\int_\Gamma ({ f(\bu^{n})}-f(\bu^{n-1})) \partial_x\beta^n dx.$$
Thus, since $f$ is Lipschitz continuous there exists $L>0$ such that
\begin{align*}
\displaystyle\frac{1}{2}\int_\Gamma |\beta^n(t)|^2 dx +\int_0^t\int_\Gamma |\partial_x\beta^n |^2 dx ds
&\leq\frac 12\int_0^t\int_\Gamma L^2|{ \bu^{n}}-\bu^{n-1}|^2 dxds+\frac 12 \int_0^t \int _\Gamma |\partial_x\beta^n |^2dxds
 \\
& \leq \frac{L^2}2\int_0^t\int_\Gamma |\beta^{n-1}|^2 dx ds + \frac 12\int_0^t\int_\Gamma |\partial_x\beta^n |^2 dx ds.\nonumber
\end{align*}
This gives that for any $t>0$  it holds that,
\begin{equation}
  \label{energy est}
\int_\Gamma |\beta^n(t)|^2 dx +\int_0^t\int_\Gamma |\partial_x\beta^n (s)|^2 dx ds\leq
{L^2}\int_0^t\int_\Gamma |\beta^{n-1}(s)|^2 dx ds.
\end{equation}
In particular, for any $T>0$ we get
$$\|\beta^n\|^2_{{C}([0,T]; L^2(\Gamma))}\leq TL^2\|\beta^{n-1}\|^2_{{C}([0,T]; L^2(\Gamma))}.$$
Choosing $TL^2 < \frac{1}{4}$,  we obtain
\begin{equation}\label{estt}
\|\beta^n \|_{{C}([0,T]; L^2(\Gamma))}\leq \frac{1}{2} \|\beta^{n-1}\|_{{C}([0,T]; L^2(\Gamma))}
\leq \frac{1}{2^{n-1}} \|\beta^1\|_{{C}([0,T]; L^2(\Gamma))},
\end{equation}
  which allows us to conclude that
$\{\bu^n\}_{n\geq 0}$   is a Cauchy sequence in $ {C}([0,T]; L^2(\Gamma)).$
Going back to \eqref{energy est} we get
\begin{align}\label{est}
\displaystyle\int_0^T\int_\Gamma |\partial_x\beta^n|^2 dx dt &\leq L^2 \displaystyle\int_0^T\int_\Gamma |\beta^{n-1}|^2 dx dt
 \leq L^2 T \|\beta^{n-1}\|^2_{{C}([0,T]; L^2(\Gamma))}\\
& \leq \frac{L^2 T}{2^{2(n-1)}}\|\beta^1\|^2_{{C}([0,T]; L^2(\Gamma))}.\nonumber
\end{align}
It follows that
$\{\partial_x \bu^n\}_{n\geq 0}$  is a Cauchy sequence in $L^2 (0,T; L^2(\Gamma)).$
Thus,
\begin{align}\label{cauchy u}
\{\bu^n\}_{n\geq 0} \mbox{ is a Cauchy sequence in } {C}([0,T]; L^2(\Gamma))\cap L^2 (0,T; V).
\end{align}
Moreover, from \eqref{aprox1}  since $f$ is Lipschitz  we find
\begin{align*}
\Big|\left\langle\displaystyle\partial_t\beta^n(t),\bpsi\right\rangle_{V',V}
\Big|\leq \|\partial_x\beta^n(t)\|_{L^2(\Gamma)}\|\partial_x\bpsi \|_{L^2(\Gamma)}
+ L\|\beta^{n-1}(t)\|_{L^2(\Gamma)}\|\partial_x\bpsi \|_{L^2(\Gamma)},
\end{align*}
i.e.,
\begin{align*}
\Big\|\partial_t\beta^n(t)\Big\|_{V'}\leq \|\partial_x\beta^n(t)\|_{L^2(\Gamma)} + L \|\beta^{n-1}(t)\|_{L^2(\Gamma)}.
\end{align*}
Then, from \eqref{est} we have
\begin{align*}
\Big\|\partial_t\beta^n\Big\|_{L^2(0,T; V')}^2&\leq 2\|\partial_x\beta^n\|_{L^2(0,T; L^2(\Gamma))}^2 + 2L^2\|\beta^{n-1}\|_{L^2(0,T; L^2(\Gamma))}^2 \\
& \leq { 4L^2 \|\beta^{n-1}\|_{L^2(0, T; L^2(\Gamma))}^2}\leq
 4L^2T \|\beta^{n-1}\|_{C([0,T]; L^2(\Gamma))}^2.
\end{align*}
In view of the estimates for $\beta^n$ in \eqref{estt} we conclude that
\begin{align}\label{cauchy ut}
\{ \partial_t\bu^n \}_{n\geq 0} \mbox{ is a Cauchy sequence in } L^2(0,T; V').
\end{align}
Since $f$ is a Lipschitz continuous function, \eqref{cauchy u}-\eqref{cauchy ut} guarantee that we can
pass to the limit in \eqref{aprox} to obtain a weak solution in the time interval $(0,T)$
in the sense of \eqref{weak-solution}. Repeating the same arguments on any interval $[nT,(n+1)T]$ we
 obtain the
existence of a global solution.

\textbf{Step II. $C^1(\rr)$-nonlinearities and $\bu_{0}\in L^2(\Gamma)\cap L^\infty(\Gamma)$.}
Following classical arguments in conservation law's theory (see for example \cite[p.~60]{MR1304494}) we truncate the nonlinearity.
Let us choose $\um \leq 0\leq  \om$ such that  $\bu_0\in [\um,\om]$ and
\[
(n-m)f(\um)\geq 0\geq (n-m)f(\om).
\]
 We introduce a  smooth function which satisfies
\begin{equation*}
\theta_{\um,\om}(r)=\begin{cases}
1,& \mbox{if }  \um\leq r\leq \om,\\
0,& \mbox{if } r \notin (\um-1,\om + 1),
		\end{cases}
\end{equation*}
and set $f_{\um,\om}(r)=\theta_{\um,\om}(r)f(r)$. This new function is globally Lipschitz and
by the previous case, there exists a global solution $\bu\in W(0,T)$ for any $T>0$  satisfying for all $\psi\in V$
\begin{equation*}\label{aprox M}
	\begin{cases}
\left\langle\displaystyle\partial_t\bu(t),\bpsi\right\rangle_{V',V} + (\partial_x \bu(t), \partial_x\bpsi)_{L^2(\Gamma)}
=  ( f_{\um, \om}(\bu(t)) ,\partial_x\bpsi),
\mbox{ a.e. on } [0,T],\\[6pt]
\bu(0)=\bu_0.
		\end{cases}
\end{equation*}
Let us prove that  $\bu(t)\in [\um, \om]$ for all $t>0$.
The function $f_{\um,\om}$ introduced above
also satisfies the   hypothesis
\[
(n-m)f_{\um,\om}(\um)\geq 0\geq (n-m)f_{\um,\om}(\om).
\]
Using Corollary \ref{particular.apriori}  with $f_{\um,\om}$ we obtain
$\um \leq \bu(t) \leq \om $ for all $ t>0 $.  Since $f(\bu)=f_{\um,\om}(\bu)$,  for $\bu$ lying between $\um$ and $\om $  we obtain that $\bu$ satisfies for any $T>0$ and for all  $\bpsi\in V$
\begin{equation}
	\begin{cases}
\left\langle \partial_t\bu(t),\bpsi\right\rangle_{V',V} + (\partial_x \bu(t), \partial_x\bpsi)_{L^2(\Gamma)}
= ( f(\bu(t)) ,\partial_x\bpsi) \mbox{ a.e. on} [0,T],\\[6pt]
\bu(0)=\bu_0.
\end{cases}\nonumber
\end{equation}
This proves  the existence of  a global solution in this case.

{\bf Step III. Uniqueness.}  We consider $\bu_1$ and $\bu_2$ in $L^\infty((0,\infty)\times \Gamma)$ two solutions of problem \eqref{weak-solution} and define $\beta=\bu_1-\bu_2$. It follows that $\beta(0)=0$ and the same arguments used in Step I show that $\beta$ satisfies
\[
\frac{1}2 \frac{d}{dt}\int _\Gamma \beta^2 dx+\int_{\Gamma} |\partial_x\beta|^2 dx=
\int _{\Gamma} (f(\bu_1)-f(\bu_2))\partial_x\beta dx.
\]
Let us denote $ \widetilde M=\|\bu_1\|_{L^\infty((0,\infty)\times \Gamma)}+\|\bu_2\|_{L^\infty((0,\infty)\times \Gamma)}$. Function  $f|_{[- \widetilde M, \widetilde M]}$ is Lipschitz, so it satisfies
$
|f(\bu_1)-f(\bu_2)|\leq L|\bu_1-\bu_2|
$
for some positive constant $L=L(- \widetilde M, \widetilde M)$.
Then,  we have
\begin{align*}
\label{}
  \frac 12 \frac{d}{dt}\int _\Gamma \beta^2 dx+\int_{\Gamma} |\partial_x\beta|^2 dx\leq
 L \int _{\Gamma}  |\bu_1-\bu_2||\beta_x| dx\leq \frac{L^2}4\int_\Gamma \beta^2 dx+\int _\Gamma |\partial_x\beta|^2 dx.
\end{align*}
By Gronwall inequality, since $\beta(0)\equiv 0$ we obtain that $\beta(t)\equiv 0$ for all $t>0$, so $\bu_1\equiv \bu_2$ and the uniqueness is proved.
\end{proof}

Let us now prove Theorem \ref{well-posedness}.

\begin{proof}[Proof of Theorem \ref{well-posedness}]
	Let us consider $\bu_0\in L^1(\Gamma)\cap L^\infty(\Gamma)$.  From Theorem \ref{well-posedness-L2}   we know  that,   there exists a unique solution $\bu\in C([0,\infty); L^2(\Gamma))\cap L^2((0,\infty); V)\cap  L^\infty((0,\infty)\times \Gamma)$ of problem \eqref{weak-solution}.  Let us show that when the initial data is in $L^1(\Gamma)$ the solution remains in $L^1(\Gamma)$ for all time $t>0$. Take a convex approximation $\rho_\eps$   in Lemma \ref{ro}  of the function $|\cdot|$ (see the proof of Corollary \ref{particular.apriori}) and consider $\eta\in H^1(\rr)$ with   $0\leq \eta\leq 1$. Set $\boldeta=(\eta_k)_{k=1}^{m+n}$ where $\eta_k=\eta$.
	Let us choose $T>0$ and $\rho_\eps(\bu)\boldeta$ as a test function in the weak formulation \eqref{weak-solution}:
\begin{align*}
\int_{\Gamma} &\rho_\eps(\bu(T))\boldeta(x) dx-\int_{\Gamma} \rho_\eps(\bu_0)\boldeta(x) dx= \int_0^T \Big( -\big(\partial_x\bu, (\rho'_\eps(\bu )\boldeta)_x\big)_{L^2(\Gamma)} +  ( f(\bu)  ,(\rho'_\eps(\bu )\boldeta)_x \big)_{L^2(\Gamma)} \Big)dt\\
\\
&= \int_0^T \Big( -\big(\partial_x\bu, (\rho'_\eps(\bu )\boldeta)_x\big)_{L^2(\Gamma)} + \big( f(\bu)-f(0) ,(\rho'_\eps(\bu )\boldeta)_x \big)_{L^2(\Gamma)}+(n-m)f(0)\eta(0)  \rho_\eps'(\bu(t,0)) \Big)dt\\
&= \int_0^T \Big( -(\partial_x\bu, \rho_\eps''(\bu)(\partial_x\bu) \boldeta ) -\big(\partial_x\bu, \rho'_\eps(\bu )\boldeta_x\big)_{L^2(\Gamma)}+\big((f(\bu)-f(0))\rho_\eps''(\bu)\bu_x,\boldeta\big)_{L^2(\Gamma)} \\
&\qquad \qquad+ 	\big(f(\bu)-f(0),\rho_\eps'(\bu)\boldeta_x\big)_{L^2(\Gamma)}
+(n-m)f(0)\eta(0)  \rho_\eps'(\bu(t,0))\Big)dt\\
&= I_\eps+ II_\eps +III_\eps+IV_\eps +V_\eps.
\end{align*}
The first term satisfies $I_\eps\leq 0$.
	For the second and  the fourth  term we easily obtain
	\[
	| II_\eps |+ | IV_\eps  |\leq 2T^{1/2} \|\bu\|_{L^2(0,T,V)} \|\boldeta_x\|_{L^2(\Gamma)}.
	\]
	Since $\rho_\eps(s)\rightarrow \sgn(s)$ as $\eps\rightarrow 0$
we obtain that
	\[
	\lim_{\eps\rightarrow 0} V_\eps = (n-m) f(0)\eta(0)\int_0^T\sgn(\bu(t,0))dt.
	\]
	Defining $H_\eps(t)=\int_0^t (f(s)-f(0))\rho_\eps''(s)ds$, the second term satisfies	\begin{align*}
	 III_\eps=(n-m) \int_0^T H_\eps(\bu(t,0))\eta(0)dt -  \int_0^T (H_\eps(\bu(t)),\boldeta_x)dt.
	\end{align*}
	Since $|f(s)-f(0)|\leq C|s|$ on bounded intervals we have that $H_\eps(t)\rightarrow 0$ for any $t\in (0,T).$
	Moreover $|H_\eps(\bu(t,0))|\leq |\bu(t,0)|\leq \|\bu(t)\|_{V}$ so the first term in the right hand side of $III_\eps$ goes to zero. The second term can be bounded similarly as $IV_\eps$ since
	\[
	| (H_\eps(\bu(t)),\boldeta_x) |\leq \|H_\eps(\bu(t)) \|_{L^2(\Gamma)}
 \|\boldeta_x\|_{L^2(\Gamma)}\leq \| \bu(t) \|_{L^2(\Gamma)}
 \|\boldeta_x\|_{L^2(\Gamma)}.
	\]
	
	Putting together the above estimates and letting $\eps\rightarrow 0$ we obtain
	\begin{align*}
\int_{\Gamma} &|\bu(T)|\boldeta(x) dx\leq \int_{\Gamma}|\bu_0|\boldeta(x) dx+
3T^{1/2} \|\bu\|_{L^2(0,T,V)} \|\boldeta_x\|_{L^2(\Gamma)}+|(n-m)Tf(0)\eta(0)|.
 \end{align*}
Let us now choose $\boldeta =\boldeta_R  =(\psi(\cdot/R))_{k=1}^{m+n}$ with $\psi\in C^\infty_c(\rr)$, $\psi\geq 0$, satisfying $\psi \equiv 1$ in $(-1,1)$ and $\psi\equiv 0$ in $|x|>2$. It follows that $\|\boldeta_x\|_{L^2(\Gamma)}\simeq R^{-1/2}$ and letting $R$ go to infinity  we get $\bu(T)\in L^1(\Gamma)$ and 
\begin{align*}
\int_{\Gamma} |\bu(T)|dx&\leq \int_{\Gamma} |\bu_0|  dx+|(n-m)Tf(0)|.
 \end{align*}

The same arguments holds by taking $\bu\boldeta$, i.e. $\rho(s)=s$, as test function in the weak formulation \eqref{weak-solution}. Choosing  the truncation function 
$\boldeta_R=(\psi(\cdot/R))_{k=1}^{m+n}$ as above we get
\[
\int_{\Gamma} \bu(T)\boldeta_Rdx=\int_{\Gamma}\bu_0\boldeta_Rdx + O(R^{-1/2})+(n-m)Tf(0).
\]	
Letting $R$ to go to infinity will give us the mass conservation property under the assumption $(n-m)f(0)=0$.
	
	Let us now choose $\eta(x)= \psi^N(x/R)$ with $\psi^N\in C^1(\rr)$, $0\leq \psi\leq 1$, satisfying  $\psi^N\equiv 1$ in $2<|x|<N$ and supported in the set $\{1<|x|<N+1\}$,
	such that $\|\psi^N_x\|_{L^2(\rr)}$ is independent of parameter $N$. We obtain
\begin{align*}
\int_{\Gamma,2R< |x|<NR} &|\bu(T)| dx\\
&\lesssim \int_{\Gamma, R<|x|<(N+1)R} |\bu_0|  dx+
  T^{1/2} R^{-1/2} \|\bu\|_{L^2(0,T,V)}  \|\psi^N_x\|_{L^2(\rr)}+|(n-m)Tf(0)|\\
&  \lesssim \int_{\Gamma, R<|x| } |\bu_0|  dx + T^{1/2}R^{-1/2}  +|(n-m)Tf(0)|.
 \end{align*}	
 Letting $N$ going to infinity we obtain a tail control of the solution in terms of the tail control of the initial data:
 \[
 \int_{\Gamma,2R< |x|} |\bu(T)| dx \lesssim \int_{\Gamma, R<|x| } |\bu_0|  dx + T^{1/2}R^{-1/2}  +|(n-m)Tf(0)|.
 \]
The tail control together with the $L^1(\Gamma)$ contraction property imply that the solution $\bu$ not only belongs to $ C([0,\infty),L^2(\Gamma))$ but also  to the space $C([0,\infty),L^1(\Gamma))$.

		Finally, the $L^1(\Gamma)$-contraction property follows as in case {\it iii)} of Corollary \ref{particular.apriori}. Similar computations to establish the contraction property
	 	 has been done in \cite[Th.~1.2]{MR2679594} without concerning regularity assumptions on the solutions or restrictions on nonlinearity.
\end{proof}

\begin{proof}[Proof of Theorem \ref{decay.estimates}]
The mass conservation follows from the fact that $f(0)=0$.
	For the second estimate, choosing $\rho(s)=|s|^{p}$, $p\geq 2$ in Lemma \ref{ro} and using that $(n-m)a\geq 0$,
	we obtain that
\[
\frac {d}{dt} \int_{\Gamma} |\bu|^pdx +\frac{4(p-1)}p\int_{\Gamma} |\partial_x |\bu|^{p/2}|^2dx\leq 0.
\]
For $p=2$ this gives us the energy estimate.
Using Nash-like inequalities for half-lines (see for example \cite{haeseler} in the case of graphs having some infinite edges or simple use even extensions and use the classical ones for the real line)
we can follow the arguments in \cite[Section 4.2]{MR1124296} to obtain the desired decay in any $L^p$-norm. We emphasize that tracking the constants as in  \cite[Proof of Th.4.2]{MR1124296} the decay is obtained also in the $L^\infty(\Gamma)$ norm.
\end{proof}

\section{Self similar profiles}\label{self-similar}
Let us now discuss the self-similar profile $\bu_M(t)=(u_{M,k}(t))_{k=1}^{m+n}$ that appears in the characterization of the long time behavior of our solutions in Theorem \ref{first.term}.
In the case $q>2$ the self-similar profile  is a  Gaussian  given by
	 \begin{equation}\label{profile_asy}
	u_{M,k}(t,x)=\frac{2M}{m+n}\frac {1}{\sqrt {4\pi t}}e^{-\frac{x^2}{4t}}, \qquad x\in I_k.
	\end{equation}
We have that this profile $\bu_M$ satisfies the heat equation on each half line together with some coupling conditions:
\begin{equation}
	\label{heat.dirac}
	\begin{cases}
		\partial_t u_i(t, x)-\partial_{xx}u_i(t, x)=0,& t>0,x<0, i\in \{1,\dots,n\},\\
		\partial_t u_j(t, x)-\partial_{xx}u_j(t, x)=0,& t>0,x>0, j\in \{n+1,\dots,n+m\},\\
		u_i(t,0)=u_j(t,0), & i,j\in\{1,\dots,n+m\},\\
		\displaystyle\sum_{i=1}^n  \partial_x u_i(t,0) \displaystyle=\sum _{j=n+1}^{n+m} \partial_x u_j(t,0).
	\end{cases}
\end{equation}
Moreover, it is   straightforward  that $\bu_M $ satisfies
\begin{align}\label{limit.eq.def.2}
\int _0^\infty \int_{\Gamma}\bu_M (t,x) (\partial_t \bp+\partial_{xx}\bp)dxdt+M \bp(0,0) =0,
	\end{align}
for any $\bp=(\varphi_k)_{k=1}^{n+m}\in C_c([0,\infty), D(\Delta_\Gamma))\cap C^1_c([0,\infty),L^2(\Gamma))$.

In \cite[Proof of Th.~3.1, Step II]{rossi2020} it has been proved the following existence and uniqueness result.

\begin{theorem} \label{teo_heat}For any  $M$ there exists a unique function $\bu_M\in L ^1_{loc}((0,\infty); L^1(\Gamma))$ satisfying \eqref{limit.eq.def.2}  and it is given by  the self-similar profile   \eqref{profile_asy}.
	\end{theorem}

Let us now consider the system \eqref{main.eq} with the nonlinearity $f(u)=-|u|^{q-1}u$, $q>1$, that is,
 \begin{equation}
	\label{heat.dirac.2}
	\left\{\begin{array}{l}
		\partial_t u_i(t,x)-\partial_{xx}u_i(t,x)-\partial_x(|u|^{q-1}u)=0, \qquad t>0,\ x<0, \ i\in \{1,\dots,n\},\\[5pt]
		\partial_t u_j(t, x)-\partial_{xx}u_j(t,x)-\partial_x(|u|^{q-1}u)=0, \qquad t>0,\ x>0, \ i\in \{n+1,\dots,n+m\},\\ [5pt]
		u_i(t,0)=u_j(t,0), \qquad t>0,\ i, j\in\{1,\dots,n+m\},\\ [5pt]
		\displaystyle \sum_{i=1}^n  \left(\partial_x u_i(t,0)+ |u_i|^{q-1}u_i(t,0)\right)=\sum _{j=n+1}^{n+m} \left(\partial_x u_j(t,0)+|u_j|^{q-1}u_j(t,0)\right), \ t>0,
	\end{array}\right.
\end{equation}
with the initial datum given by
\begin{equation}\label{initial.delta.2}
	\lim _{t\rightarrow 0}\sum _{i=1}^{n}\int _{-\infty}^0 u_i(t,x)\varphi_i(x)dx+\sum _{j=n+1}^{n+m}\int_0^{\infty} u_j(x)\varphi_j(t, x)dx=M\bp (0)
\end{equation}
for any $\bp=(\varphi_l)_{l=1}^{n+m}\in BC(\Gamma)$, with $\bp(0):=\varphi_i( 0)=\varphi_j( 0)$  for any $\ i,j\in \{1,\dots,n+m\}.$

The well-posedness of the above system is related to the well-posedness to the following  problem in the real line
\begin{equation}
	\label{conv.dirac}
	\begin{cases}
		\partial_t u-\partial_{xx}u-\partial_x(|u|^{q-1}u)=0,& t>0,x\neq 0,  \\
		u(t,0-)=u(t,0+), & t>0,\\
		n ( \partial_x u +|u|^{q-1}u)(t,0-)=m ( \partial_x u +|u|^{q-1}u)(t,0+),& t>0,
	\end{cases}
\end{equation}
with the initial datum $u_0$ satisfying $(n1_{\rr_-}+m1_{\rr_+})u_0=M\delta_0$  in the sense of bounded measures:
\begin{equation}
	\label{initial.delta.4}
	\lim _{t\rightarrow 0}n\int _{-\infty}^0 {  u(t,x)}\varphi (x)dx+m{ \int_0^{\infty}} u(t,x)\varphi(x)dx=M\varphi(0)
\end{equation}
for any $\varphi\in BC(\rr)$.

When $n=m$ this corresponds to the classical convection-diffusion equation posed in the real line.
In the case of the real line and $q>1$ the case of nonnegative/nonpositive solutions was considered in \cite{MR1233647} where the existence and uniqueness of the solutions of the problem
\[
u_t=u_{xx}+(|u|^{q-1}u)_x
\]
with a Dirac delta as initial datum was proved. Later, in \cite{MR1440033}   the well-possedness
in the class of solutions that may change sign has been proved for $q\in (1,2]$ by assuming  a uniform tail control as time goes to zero,
that is, for any fixed $r>0$
\[
 \int_{|x|>r} |u(t,x)|dx \rightarrow 0, \ \text{as}  \ t\rightarrow 0.
\]

Similar questions for the system \eqref{conv.dirac}-\eqref{initial.delta.4} in the case $n\neq m$ still remain to be analyzed.

In the following we will give a positive answer of the above question in the particular case when
$f(u)=-|u| u$.
We analyze the solutions to the  system \eqref{heat.dirac.2}
where the initial datum is taken in the sense of measures \eqref{initial.delta.2}.
Nonlinearity $f(s)=-|s|s$ fits in the hypothesis of Theorem \ref{decay.estimates} under the assumption $n\geq m$.
 Our main goal is to show that there is a unique solution to  \eqref{heat.dirac.2}-\eqref{initial.delta.2}   and obtain a self-similar expression for the components.
Let us first comment on the main ideas used in the proof. We consider the case of nonnegative solutions, the case of nonpositive solutions being similar. We will prove that, under the sign condition on the solutions, all the components defined on
the same interval $\rr_{-}$ or $\rr_+$ are equal: i.e. $u_i=u_{i'}$ for all $1\leq i, i'\leq n$ and $u_{n+j}=u_{n+j'}$ for all $1\leq j,j'\leq m$. Hence, the problem of existence and uniqueness of solutions of system \eqref{heat.dirac.2}-\eqref{initial.delta.2} is reduced to study the same problem for a
related equation in the real line:
given two positive numbers $n$ and $m$ consider the following coupled Burgers' equations
\begin{equation}
	\label{coupled-burgers}
	\left\{
	\begin{array}{ll}
		\partial_t u -\partial_{xx}u-\partial_x(|u|u)=0,& t>0,x\neq 0, \\[6pt]
		u(t,0-)=u(t,0+), &  t>0,\\[6pt]
		 n ( \partial_x u +|u|u)(t,0-)=m  (\partial_x u+|u|u)(t,0+), &  t>0,
	\end{array}
	\right.
\end{equation}
with the initial datum taken in the sense of measures
\begin{equation}
{ 	\label{initial.delta.3}
	\lim _{t\rightarrow 0}n\int _{-\infty}^0 u(t,x)\varphi (x)dx+m\int_0^{\infty} u(t,x)\varphi(x)dx=M\varphi(0), }
\end{equation}
for any   $\varphi\in BC(\rr)$.

\begin{theorem}
	\label{burgers.nm}Let $n\geq m$.
	For any positive/negative number $M$ there exists a unique nonnegative/nonpositive solution $u_M\in C((0,\infty),L^1(\rr))$ to the
	system \eqref{coupled-burgers} with the initial condition \eqref{initial.delta.3}.
	This unique solution is given by $$u_M(t,x)=\frac 1{\sqrt t}G_M(\frac x{\sqrt {t}})$$ where
	\begin{equation}
\label{perfil.fm}
  G_M(y)= \frac{\alpha_{n,m,M} e^{-\frac{y^2}4}}{1+\alpha_{n,m,M}\int_{-\infty}^y e^{-s^2/4}ds}
\end{equation}
	and $\alpha_{n,m,M} $ is the unique solution in the interval $ (-\frac 1{2\sqrt \pi}, \infty)$  of the   equation
	\begin{equation}
	\label{eq.alpha}
	 |1+\alpha \sqrt{ \pi} |^{n-m}|1+2\alpha \sqrt\pi|^m=e^M.
	\end{equation}
\end{theorem}

\begin{proof}[Proof of Theorem \ref{burgers.nm}]
{\bf Step I. Existence.}  	Looking for self similar solutions $$u(t,x)={  \frac{1}{\sqrt{t}}}G_M(\frac x{\sqrt {t}})$$
	we  obtain that the profile $G_M$ should satisfy
	\[
	G_M(y)=
	\left\{
	\begin{array}{ll}
\displaystyle \frac{\alpha e^{-\frac{y^2}4}}{1+\alpha\int_{-\infty}^y e^{-s^2/4}ds}, & y<0,\\
	\displaystyle	\frac{-\beta e^{-\frac{y^2}4}}{1+\beta\int_y^\infty e^{-s^2/4}ds}, & y>0.
	\end{array}
	\right.
	\]
	The coupling conditions at $x=0$ impose that $\frac 1\alpha +\frac 1\beta =-2\sqrt \pi$.  Notice that the two branches of $G_M$ are of the same  form \eqref{perfil.fm}.  The mass conservation gives us that $\alpha$ is a solution to
	\[  |1+\alpha \sqrt \pi|^{n-m}|1+2\alpha \sqrt\pi|^m=e^M.
	\]
	In order to guarantee that $u(t)\in L^1(\rr)$, so $G_M\in L^1(\rr)$ we have to look for $\alpha\in (-\frac 1{2\sqrt \pi}, \infty)$. Elementary computations shows that the left hand side of the equation is an increasing function of $\alpha$ in the interval $ (-\frac 1{2\sqrt \pi}, \infty)$ and then there exists an unique $\alpha_{n,m,M}$ solution of the above equation and hence an unique  profile $G_M\in L^1(\rr)$.

	\medskip
	
	{\bf Step II. Uniqueness.}
	The proof of this fact is quite technical and we divide it in some steps. We consider the case of
	nonnegative solutions of equation \eqref{coupled-burgers} with initial data \eqref{initial.delta.3}.
	The case of nonpositive solutions can be treated similarly
	by replacing $u$ by $-u$.

{\textit{Step a). Hopf-Cole transform}.}
		  Fix an initial data $u_0\in L^1(\rr) $ with mass $M$.
		Since $u_{0 }\in L^1(\rr) $  the results  in Theorem \ref{decay.estimates} and Remark \ref{l1solutions} give us that $u \in C([0,\infty),L^1(\rr))\cap L^2_{loc}((0,\infty),H^1(\rr))\cap L^\infty_{loc}((0,\infty)\times\rr)$. Its total mass is conserved along time
\[
M :=\int _{\rr} u (t,x)(n1_{\rr_-}+m1_{\rr_+})dx=\int _{\rr}u_{0 }(x)(n1_{\rr_-}+m1_{\rr_+})dx.
\]
   Also, $ u$ satisfies the decay obtained in Theorem \ref{decay.estimates}:
\begin{equation}
\label{decay.u.123}
\|u (t)\|_{L^\infty(\rr)}\leq     C  t^{-1/2}M.
\end{equation}
		Since $u \in   L^2_{loc}((0,\infty),H^1(\rr))$ we have that for a.e. $t>0$ the map $t\mapsto u (t,0)$ is well defined and belongs to $L^2_{loc}((0,\infty))$. By \eqref{decay.u.123} we also have that for a.e. $t>0$,
\begin{equation}
\label{decay.u.0}
	0\leq u (t,0) \leq  C  t^{-1/2}M.
\end{equation}

		For $u$ solution of   system \eqref{coupled-burgers} with initial data $u_0 $
we introduce the function
	\[
	v (t,x) =\int_{-\infty}^x u (t,y)(n1_{\rr_{-}}+m1_{\rr_+})dy.
	\]
	Since $u(t)\in L^1(\rr)$ function $v$ is continuous, in particular $v(t,0-)=v(t,0+)$.
	It also follows that $0\leq v\leq M$ and
	$$v_x (t,x) =u (t,x)(n 1_{\rr_{-}} (x)+m1_{\rr_+} (x)).$$
 We also introduce the function $w $  given by
	\[
	w(t,x)=\left\{
\begin{array}{ll}
	\displaystyle e^{v (t,x)/n},& x<0,\\
	e^{v (t,x)/m},&x>0.
\end{array}
	\right.
	\]	
Under the assumption $n\geq m$ we have that for fixed $t>0$ the map $x\rightarrow w(t,x)$ is increasing and satisfies
 $1\leq w \leq e^{M/m}$. Using that
 $w_x=wu$ we obtain that  $w$ satisfies the following system of coupled heat equations
	\begin{equation*}
	\left\{
	\begin{array}{ll}
		\partial_t w   -\partial_{xx}w =0,& t>0,x\neq 0, \\[6pt]
		w^n(t,0-)=w^m(t,0+), &  t>0,\\[6pt]
		\displaystyle  \frac{w_{x}}{w}(t,0-)= \frac{w_{x}}{w}(t,0+)=u(t,0), &t>0,\\ [6pt]
		 w(0,x)=w_0(x).
 	\end{array}
	\right.
\end{equation*}
Let us denote $w_R$ and $w_L$ the restrictions  $w$ to $\rr_+$ and $\rr_-$  respectively.
Also, $w_{0R}$ and $w_{0L}$ are the restrictions of the initial datum to $(0, \infty)$ and $(-\infty, 0)$ respectively.
Solving separately the two heat equations with Neumann boundary condition on half lines we have the following representation
formulas
\begin{equation}\label{w_Rx}
w_R(t,x)=\int_0^\infty (K_t(x-y)+K_t(x+y))w_{0R}(y)dy-2\int_0^t K_{t-\tau}(x)w_{R,x}(\tau,0)d\tau,
\end{equation}
and
\begin{equation}\label{w_Lx}
w_L(t,x)=\int_{-\infty}^0 (K_t(x-y)+K_t(x+y))w_{0L}(y)dy+2\int_0^t K_{t-\tau}(x)w_{L,x}(\tau,0)d\tau.
\end{equation}
Writing both expressions at $x=0$ we get
\begin{align}\label{wr}
w_R(t,0)=2\int_0^\infty K_t(y)w_{0R}dy-\frac 1{\sqrt \pi}\int_0^t \frac{w_{R}(\tau,0)u(\tau,0)}{\sqrt{t-\tau}}d\tau
\end{align}
and
\begin{align}\label{wl}
w_L(t,0)=2\int_{-\infty}^0 K_t(y)w_{0L}dy+\frac 1{\sqrt \pi}\int_0^t \frac{w_{L}(\tau,0)u(\tau,0)}{\sqrt{t-\tau}}{ d\tau.}
\end{align}
We now introduce the functions $\eta_R(t)=\sqrt t u(t,0)w_R(t,0)$ and $\eta_L(t)=\sqrt t u(t,0)w_L(t,0)$.
In view of \eqref{decay.u.0} we have
\[
0\leq \eta_L(t)  \leq  C   Me^{M/m}, \qquad 0\leq \eta_R(t)  \leq  C   Me^{M/m}.
\]
Both of them are bounded functions in $[0,\infty)$ and they satisfy the system
\[
\eta_R(t)=\sqrt t u(t,0)(M_{0R} (t)-(\mathcal{L}\eta_R)(t) ),
\]
\[
\eta_L(t)=\sqrt t u(t,0)(M_{0L}(t) +(\mathcal{L}\eta_L)(t) ),
\]
where \[
M_{0R}(t)=2\int_0^\infty K_t(y)w_{0R}(y)dy,\ 
 M_{0L}(t)=2\int_{-\infty}^0 K_t(y)w_{0L}(y)dy\]
  and
\[
(\mathcal{L}\eta)(t)=\frac 1{\sqrt \pi}\int _0^t\frac{\eta(s)}{\sqrt{t-s}\sqrt s}ds.
\]
It follows that
\begin{equation}\label{eta1}
\frac{\eta_R(t)}{M_{0R}(t) -(\mathcal{L}\eta_R)(t) }=\frac{\eta_L(t)}{M_{0L} (t)+(\mathcal{L}\eta_L)(t)}.
\end{equation}
We now use the coupling conditions \eqref{wr} and  \eqref{wl} on $w$  to close the system for $\eta$.
The identity $w_L^n(t,0)=w^m_R(t,0)$ implies that
\begin{equation}\label{eta2}
(M_{0R}(t)-(\mathcal{L}\eta_R)(t))^m=(M_{0L}(t)+(\mathcal{L}\eta_L)(t))^n .
\end{equation}
From \eqref{eta1} we get
\begin{equation}
\label{rel.etas}
\eta_R(t)=\eta_L(t)(M_{0L}(t)+(\mathcal{L}\eta_L)(t))^{\frac nm-1}.
\end{equation}
Introducing this expression in \eqref{eta2} we obtain that $\eta_L$ satisfies  the following integral equation
\begin{equation}\label{eq.eta}
\Big(M_{0L}(t)+(\mathcal{L}\eta_L)(t)\Big)^{\frac nm}+\mathcal{L} \Big(\eta_L(t)\big(M_{0L}(t)+(\mathcal{L}\eta_L)(t)\big)^{\frac nm-1}  \Big)=M_{0R}(t), \quad \forall t>0.
\end{equation}
When $M_{0L}(t)\equiv 1$ and $M_{0R}(t)\equiv e^{M/m}$ there exists a unique continuous function $\eta:[0,T]\rightarrow \rr$
 that solves equation \eqref{eq.eta},  $\eta \equiv  \alpha_{n,m,M}$ for some positive constant $\alpha_{n,m,M}$. This is a consequence of the fact that  $\mathcal{L}$ applied to the function identically one gives $\mathcal{L} ({\bf 1})\equiv {\bf \sqrt {\pi}}$ and is monotone, i.e., for any $\eta_1(s)\leq \eta_2(s)$ in $(0,T)$ we have
$\mathcal{L}(\eta_1)(t)\leq \mathcal{L}(\eta_2)(t)$ for all $t\in (0,T)$. However, the analysis of solutions of \eqref{coupled-burgers}-\eqref{initial.delta.3} is more involved since we do not have enough regularity for the solutions in order to guarantee that the corresponding $\eta$ is a continuous function.

\textit{Step b). Approximation.}  Let $M>0$ and $u\in C((0,\infty),L^1(\rr))$ a solution of \eqref{coupled-burgers}-\eqref{initial.delta.3}.
We fix $R>0$ and  choose   a truncation of the solution $u$ at time $\eps$:  $u_{0,\eps}(x)= 
u(\eps,x){\bf 1}_{\{|x|<R\}}(x)$. We solve   equation \eqref{coupled-burgers}  with $u_{0,\eps}\in L^1(\rr)$ initial data. In view of Remark \ref{l1solutions} there exists a unique solution $u^\eps\in C([0,\infty),L^1(\rr))$. Moreover,
 using the contraction principle it follows that $u^\eps(t)$, converges to $u(t)$ in $C((0,\infty),L^1(\rr))$.

For $u$, respectively for $u^\eps$, we construct functions $v$ and $w$, respectively $v^\eps$ and $w^\eps$, as in Step {\it a)}.
By construction $u=w_x/w$.
Since $u^\eps(t) \rightarrow u(t)$ in $L^1(\rr))$ as $\eps\rightarrow 0$ it follows that $v^\eps(t)\rightarrow v(t)$ and
$w^\eps(t)\rightarrow w(t)$ in $L^\infty(\rr)$, so $w^\eps(t,x)\rightarrow w(t,x)$ for a.e. $x\in \rr$.
The main  goal now is to prove that as $\eps\rightarrow 0$ the following holds
\begin{equation}
\label{limit.w}
w^\eps(t,x)\rightarrow
\begin{cases}
1 +\alpha_{m,n,M}  \int ^{x/\sqrt t}_{-\infty} e^{-s^2/4}ds,& x<0,\\[6pt]
(1+\alpha_{n,m,M}\sqrt \pi)^{\frac nm-1} \Big(1+  \alpha_{n,m,M}   \int_{-\infty}^{x/\sqrt t} e^{-s^2/4}ds\Big),& x>0.
\end{cases}
\end{equation}
This gives us an explicit expression for $w$ and hence $u=w_x/w$ is given by
\[
u(t,x)=\frac{1}{\sqrt t}\frac{\alpha_{n,m,M} e^{-\frac{|x|^2}{4t}}}{1+\alpha_{n,m,M}\int_{-\infty}^{x/\sqrt t} e^{-s^2/4}ds}
\]
which finishes the proof.

In the following we prove \eqref{limit.w}.
By construction the mass of $u_{0,\eps}$ denoted by $M_{\eps}$ satisfies $M_\eps\leq M$.
  Moreover,  we have that $M_\eps\rightarrow M$ as $\eps\rightarrow 0$.
Also, using the decay of $u^\eps$ we obtain that for a.e. $t>0$
\begin{equation}
\label{decay.u.eps}
|u^\eps(t,0)| \leq \|u^\eps(t)\|_{L^\infty(\rr)}\leq C  t^{-1/2} M_\eps\leq C t^{-1/2}M.
\end{equation}

	The main  goal now is to prove that $u^\eps$, the solution of Burgers' equation \eqref{coupled-burgers} with initial data $u_{0,\eps}$, converges to the
	self-similar solution  in Step I.
 By construction $1\leq w_{0,\eps}\leq e^{M/m}$ and then
\[
M_{0L}^\eps(t)=2\int _{-\infty}^0 K_t(y)w_{0\eps}(y)dy\geq 1,
\]
and
\[
M_{0R}^\eps(t)=2\int _0^{\infty}K_t(y)w_{0\eps}(y)dy\leq e^{M/m}.
\]
Following the same construction in Step a) we introduce function $\eta_L^\eps$ which satisfies
\[
0\leq \eta_L^\eps(t)\leq Me^{M/n}, \ \text{a.e.}\ t>0,
\]
and
for all $t>0$
\begin{align*}
\Big(M^\eps_{0L}(t)+(\mathcal{L}\eta^\eps_L)(t)\Big)^{\frac nm}+\mathcal{L} \Big(\eta^\eps_L(t)\big(M^\eps_{0L}(t)+(\mathcal{L}\eta^\eps_L)(t)\big)^{\frac nm-1}  \Big)  =M^\eps_{0R}(t).
\end{align*}

Let us fix $T>0$.
We extract a subsequence $\eta_L^{\eps_n}$ such that,
$\eta_L^{\eps_n}\stackrel{\ast}{\rightharpoonup}\eta$ in $L^\infty(0,T)$.
Observe that the weak star convergence implies that for any $0<t<T$ we have
\begin{align*}
{ (\mathcal{L}\eta^\eps _L)(t)}&=\frac 1{\sqrt \pi}\int_0^t \frac{\eta_L^\eps(s)ds}{\sqrt{t-s}\sqrt s}=\frac 1{\sqrt \pi}\int_0^T \frac{1_{(0,t)}(s)}{\sqrt{t-s}\sqrt s}{\eta_L^\eps(s)}ds\\
& \qquad \rightarrow
\frac 1{\sqrt \pi}\int_0^T \frac{1_{(0,t)}(s)}{\sqrt{t-s}\sqrt s}{\eta(s)}ds=\mathcal{L}(\eta)(t).
\end{align*}
Let us now consider a sequence $\|\theta_\eps\|_{L^\infty(0,T)}\leq C$ such that
 $\theta_\eps(t)\rightarrow \theta (s)$ for all $s\in (0,T)$.  Thus,
\begin{align*}
	{ (\mathcal{L}(\eta^\eps_L\theta_\eps))(t)}&=\frac 1{\sqrt \pi}\int_0^t \frac{\eta_L^\eps(s)(\theta_\eps(s)-\theta(s))ds}{\sqrt{t-s}\sqrt s}+
	\frac 1{\sqrt \pi}\int_0^t \frac{\eta_L^\eps(s)\theta(s)ds}{\sqrt{t-s}\sqrt s}\\
	&\qquad \rightarrow 0 +\frac 1{\sqrt \pi}\int_0^t \frac{\eta(s)\theta(s)ds}{\sqrt{t-s}\sqrt s}=\mathcal{L}(\eta\theta)(t),
\end{align*}
where the first convergence follows from Lebesgue dominated convergence and the second one from the weak star convergence of $\eta^\eps_L$.

We now use that for any $t>0$,  $M_{0L}^\eps(t)\rightarrow 1$ and $M_{0R}^\eps\rightarrow e^{M/m}$, so
\[
\theta_\eps(t)=\big(M^\eps_{0L}(t)+(\mathcal{L}\eta^\eps_L)(t)\big)^{\frac nm-1} 
\]  
is bounded and converges as $\eps\rightarrow 0$  to $\theta(t)=\big(1+(\mathcal{L}\eta )(t)\big)^{\frac nm-1} $. 
 It follows that the limit point $\eta$ is a nonnegative function in $\eta\in L^\infty(0,T)$  that satisfies
\begin{equation}\label{eq.int}
\Big(1+(\mathcal{L}\eta )(t)\Big)^{\frac nm}+\mathcal{L} \Big(\eta (t)\big(1+(\mathcal{L}\eta )(t)\big)^{\frac nm-1}  \Big)=e^{M/m}, \ \forall \ t\in (0,T).
\end{equation}
We claim that the unique nonnegative solution in $L^\infty(0,T)$ of the above equation is the constant function $\eta\equiv\alpha_{n,m,M}$. Since the limit point is unique    the whole sequence $(\eta_L^\eps)_{\eps>0}$ converges to $\eta$ not only
 along a subsequence.

\textit{ Step  c). Uniqueness of the nonnegative bounded solutions of $\eqref{eq.int}$}.
We construct inductively two sequences $(a_k)_{k\geq 0}$ and $(b_k)_{k\geq 0}$ such that
$$a_k\leq { (\mathcal{L}\eta) (t)}\leq b_k$$
for all $k\geq 0$ and then prove that the two sequences converge {  to the same} limit $c_{m,n,M}$. Using Laplace's transform we obtain that the unique bounded solution $\eta$ {  of the} equation ${ (\mathcal{L}\eta) (t)}=c_{m,n,M}$ for all $t>0$ is the constant function $\eta \equiv \alpha_{n,m,M} := c_{n,m,N}/\sqrt{\pi}$.

Let us construct the two sequences $(a_k)_{k\geq 0}$ and $(b_k)_{k\geq 0}$ inductively.
We start with $a_0=0$. This implies that ${ (\mathcal{L}\eta)(t)}\leq b_0$ where $b_0$ is the unique positive solution of $(1+x)^{n/m}+x=e^{M/m}.$
This gives the first terms in the inductive argument $a_0$, $b_0$.

Now, let us assume that $a_k\leq { (\mathcal{L}\eta)(t)}\leq b_k$. It follows that
\[
 \big(1+(\mathcal{L}\eta )(t)\big)^{\frac nm}+{ (\mathcal{L}\eta) (t)}  \big(1+a_k\big)^{\frac nm-1}  \leq e^{M/m}
 \leq  \big(1+(\mathcal{L}\eta )(t)\big)^{\frac nm}+{ (\mathcal{L}\eta) (t)}  \big(1+b_k\big)^{\frac nm-1}
\]
and then $a_{k+1}\leq { (\mathcal{L}\eta)(t)}\leq b_{k+1}$ where $b_{k+1}$ and $a_{k+1}$ are the unique positive solutions of the equations
\[
(1+x)^{n/m}+x (1+a_k )^{\frac nm-1}=e^{M/m}
\]
and
\[
(1+x)^{n/m}+x (1+b_k )^{\frac nm-1}=e^{M/m},
\]
respectively. Also by construction $a_0=0<a_1$ and then $b_0=b_1>b_2$. Inductively, we obtain $a_k\leq a_{k+1}$ and $b_k\geq b_{k+1}$.
Denoting their limits by $L_a$ and $L_b$ we get that $L_a\leq L_b$ and
\[
(1+L_b)^{n/m}+L_b (1+L_a )^{\frac nm-1}=e^{M/m}=(1+L_a)^{n/m}+L_a (1+L_b )^{\frac nm-1}.
\]
Since $n\geq m$ it follows that
\[
\frac{1+L_a-L_b}{1+L_b-L_a}=\Big(\frac{1+L_b}{1+L_a}\Big)^{\frac{n}m-1}\geq 1
\]
and thus $$L_a=L_b=c_{n,m,M}$$ where $c_{n,m,M}$ is the unique positive solution of
\[
(1+x )^{\frac nm-1}(1+2x)=(1+x)^{n/m}+x (1+x )^{\frac nm-1}=e^{M/m}.
\]
Clearly $\alpha_{n,m,M}=c_{n,m,N}/\sqrt{\pi}$ satisfies \eqref{eq.alpha}.

\textit{ Step  e). Proof of \eqref{limit.w}}.
From the previous steps we have $\eta_L^{\eps_n}\stackrel{\ast}{\rightharpoonup}\alpha_{n,m,M}$ in $L^\infty(0,T)$.
Using now  identity  \eqref{rel.etas} for $\eta_R^\eps$ we get
\[
\eta^\eps_R(t)=\eta^\eps_L(t)(M_{0L}^\eps(t)+(\mathcal{L}\eta^\eps_L)(t))^{\frac nm-1}
\]
and
it follows that
\[
\eta^\eps_R(t)\stackrel{\ast}{\rightharpoonup} \alpha_{n,m,M}(1+\alpha_{n,m,M}\sqrt \pi)^{\frac nm-1}\ \text{in}\ L^\infty(0,T).
\]
 Going back to the representation of \eqref{w_Lx} we obtain that
\begin{align*}
w^\eps_L(t,x)&=\int_{-\infty}^0 (K_t(x-y)+K_t(x+y))w^\eps_{0L}(y)dy+2\int_0^t K_{t-\tau}(x)w^\eps_{L,x}(0,\tau)d\tau\\
&=\int_{-\infty}^0  (K_t(x-y)+K_t(x+y))w^\eps_{0L}(y)dy+2\int_0^t K_{t-\tau}(x)\frac{\eta_L^\eps(0,\tau)}{\sqrt \tau}d\tau\\
& \qquad \rightarrow 1 +2\alpha_{n,m,M}\int _0^t K_{t-\tau}(x)\frac{d\tau}{\sqrt \tau} =1+\frac{\alpha_{m,n,M}}{\sqrt \pi}\int _0^1 e^{-\frac{x^2}{4t\sigma}}\frac{d\sigma}{\sqrt \sigma \sqrt {1-\sigma}}\\
& \qquad =1 +\alpha_{m,n,M}  \int_{-\infty}^{x/\sqrt t} e^{-s^2/4}{  ds.}
\end{align*}
{ Here we} used with $a=x/t$, $x<0$,  the following identity, known as Craig's formula \cite{craig}:
\[
\int _0^1 e^{-\frac{a^2}{4\sigma}}\frac{d\sigma}{\sqrt \sigma \sqrt {1-\sigma}}=\pi {\rm erfc}(\frac{|a|}2)=
\sqrt{\pi}\int_{|a|}^\infty e^{-s^2/4}ds=\sqrt{\pi}\int_{-\infty}^{-|a|} e^{-s^2/4}ds.
\]
Similarly
\begin{align*}
w^\eps_R(t,x) & =\int_0^\infty (K_t(x-y)+K_t(x+y))w^\eps_{0R}(y)dy-2\int_0^t K_{t-\tau}(x)w^\eps_{R,x}(0,\tau)d\tau \\
&=\int_0^{\infty}   (K_t(x-y)+K_t(x+y))w^\eps_{0R}(y)dy-2\int_0^t K_{t-\tau}(x)\frac{\eta_R^\eps(0,\tau)}{\sqrt \tau}d\tau\\
& \qquad \rightarrow e^{M/m} - \alpha_{n,m,M}(1+\alpha_{n,m,M}\sqrt \pi)^{\frac nm-1}  \int_{x/\sqrt t}^{\infty}e^{-s^2/4}ds\\
& \qquad=e^{M/m} - \alpha_{n,m,M}(1+\alpha_{n,m,M}\sqrt \pi)^{\frac nm-1} \Big(2\sqrt {\pi}- \int_{-\infty}^{x/\sqrt t} e^{-s^2/4}ds\Big)\\
& \qquad=(1+\alpha_{n,m,M}\sqrt \pi)^{\frac nm-1} \Big(1+  \alpha_{n,m,M}   \int_{-\infty}^{x/\sqrt t} e^{-s^2/4}ds\Big).
\end{align*}
This shows that \eqref{limit.w} holds and the proof is finished.
\end{proof}

\begin{theorem}\label{burgers.network.delta}Let $q=2$ and $n\geq m$.
For any positive number $M$  there exists a unique nonnegative solution $\bu_M\in C((0,\infty),L^1(\Gamma))$
satisfying \eqref{heat.dirac.2} with \eqref{initial.delta.2}.  Moreover, this solution satisfies that all the components defined on
the same intervals $\rr_{-}$ or $\rr_+$ are equal: i.e. $u_i=u_{i'}$ for all $1\leq i, i'\leq n$ and $u_{n+j}=u_{n+j'}$ for all $1\leq j,j'\leq m$.
\end{theorem}

\begin{remark} {\rm
When $M<0$ there exists e unique nonpositive solution. Indeed, $v=-u\geq 0$ solve the same equation and we reduce the problem to the case of nonnegative solutions.}
\end{remark}

\begin{proof}[Proof of Theorem \ref{burgers.network.delta}]
  Existence follows from Theorem \ref{burgers.nm}.
Hence, we concentrate in showing that for any solution $\bu_M\in C((0,\infty),L^1(\Gamma))$  its components verifies $u_i=u_{i'}$ for all $1\leq i, i'\leq n$ and $u_{n+j}=u_{n+j'}$ for all $1\leq j,j'\leq m$.

Let us choose a function $\psi\in BC(\rr)$ with $\psi(0)=0$. Taking in \eqref{initial.delta.2} as test function
 a function $\bp$ with  the $k$-th component,
$\varphi_k=\psi$, $\varphi_i=0$, for $i\neq k$, we obtain that
\[
\lim _{t \searrow 0}\int_{I_k} u_k(t,x)\psi(x)dx  =0.
\]
Let us assume that $n\geq 2$ and prove that $u_1=u_2$.  The other cases are similar.
We apply the Hopf-Cole transform in order to reduce the equations satisfied by $u_1$ and $u_2$ to the heat equation.
Let us set
\[
v_i(t,x)=\int_{-\infty}^x u_i(t,y)dy, \ i\in \{1,2\}.
\]
Since $u_i \in C((0,\infty),L^1(I_i))$ it follows that $v_i\in C((0,\infty),L^\infty(I_i)) $ and satisfy
\begin{equation*}
	\left\{
	\begin{array}{ll}
		\partial_t v_i -\partial_{xx}v_i-(v_{i,x})^2=0,& t>0,x< 0,i\in \{1,2\}, \\[8pt]
		v_1(t,0-)=v_2(t,0-), &  t>0,\\[8pt]
		v_{1,x}(t,0-)=v_{2,x}(t,0-), &t>0,\\[8pt]
		0\leq v_i(t,x)\leq M, \quad & i\in \{1,2\}.
	\end{array}
	\right.
\end{equation*}
In order to  obtain  the initial datum at $t=0$, for a given $x<0$ we choose a continuous function $\psi$ such that
$\psi(y)=1$ for $y<x$ and $\psi(y)=0$ for $y\in (x/2,0)$. Thus, using the fact that we deal with nonnegative solutions, we have
\[
v_i(t,x)=\int _{-\infty}^0 u_i(t,y) \chi _{(-\infty,x)}(y)dy\leq \int _{-\infty}^0 u_i(t,y) \psi(y)dy\rightarrow 0, \ \text{as}\ t\searrow 0.
\]
Let us take $w_i(t,x)=\exp(v_i(t,x))$.  These functions $w_i$ satisfy
\begin{equation*}
	\left\{
	\begin{array}{ll}
		\partial_t w_i -\partial_{xx}w_i=0,& t>0,x< 0,i\in \{1,2\}, \\[8pt]
		w_1(t,0-)=w_2(t,0-), &  t>0,\\[8pt]
		w_{1,x}(t,0-)=w_{2,x}(t,0-), &t>0\\[8pt]
   \lim_{t\searrow 0} w_i(t,x)={ 1},\ & x<0, i\in \{1,2\},\\[8pt]
   1\leq  w_i(t,x)\leq e^M, & t>0, \  x<0,  i\in \{1,2\}.
	\end{array}
	\right.
\end{equation*}
It follows that $\widetilde w=w_1-w_2$ is a uniformly bounded solution of the heat   equation in  $(-\infty,0)$ with   Neumann boundary condition at $x=0$ and initial datum at $t=0$ identically zero.    Then, it  follows that $w_1\equiv w_2$ and hence $u_1\equiv u_2$.

Let us now consider the function $u(t,x)$ given by  $u(t,x)=u_i(t,x)$ for $x<0$ and $u(t,x)=u_j(t,x)$ for $x>0$. The new function $u$ is defined in the real line instead of on the graph $\Gamma$ and satisfies the system \eqref{coupled-burgers} with initial datum at $t=0$   a multiple of the delta function  taken in the sense \eqref{initial.delta.3}. In view of Theorem \ref{burgers.nm}
  we have  $u(t,x)=u_M(t,x)= G_M(x/\sqrt t)/{\sqrt t}$ the unique profile $G_M$ being given by \eqref{perfil.fm}. The proof is now complete.
\end{proof}

\section{The first term in the asymptotic behavior}\label{asymptotic behavior}
In this section we  prove Theorem \ref{first.term}. To fix ideas we consider the nonlinearity $f(u)=-|u|^{q-1}u$, $q\geq 2$ and $n\geq m$.
Let   $\bu_0\in L^1(\Gamma)\cap L^\infty(\Gamma)$
 and $\bu$ the solution obtained in Theorem \ref{decay.estimates}.
As we have mentioned, in order to  prove Theorem \ref{first.term}
we proceed by using a scaling argument. We introduce the family $(\bu^\lambda)_{\lambda>1}$:
\[
u^\lambda_k(t,x)=\lambda u_k(\lambda^2 t,\lambda x), \qquad \ x\in I_k, \ t>0, \ k=1,\dots, m+n.
\]
It follows that $\bu^\lambda$ satisfies
\[
\left\{
\begin{aligned}
& \bu^\lambda\in L^2((0,T); V), \quad \partial_t\bu^\lambda\in L^2((0,T); V'), \\
 & \langle\partial_t\bu^\lambda(t),\bpsi \rangle_{V',V} + (\partial_x \bu^\lambda(t), \partial_x\bpsi)_{L^2(\Gamma)} = \lambda^{2-q}(f(\bu^\lambda(t)),\partial_x\bpsi)_{L^2(\Gamma)},
\,\,\mbox{a.e. in } \, (0,\infty),\, \forall \bpsi\in V, \\
& \bu^\lambda(0)=\lambda\bu_0(\lambda \cdot) .
\end{aligned}
\right.
\]
 In view of the estimates obtained in Theorem \ref{decay.estimates} we obtain the following uniform bounds for the family $(\bu^\lambda)_{\lambda>1}$.

\begin{proposition}
 For any $0<t_1<t_2<\infty$, the family $(\bu^\lambda)_{\lambda>1}$ satisfies
	\begin{equation}
\label{est.1}
\|\bu^\lambda \|_{L^\infty((t_1,t_2); L^p(\Gamma))}\leq C(t_1,\|\bu_0\|_{L^1(\Gamma)}), \ 1\leq p\leq \infty,	
	\end{equation}
	\begin{equation}
\label{est.2}
  \int_{t_1}^{t_2}\int_{\Gamma} |\partial_x\bu^\lambda|^2dxdt\leq t_1^{-1/2} \|\bu_0\|_{L^1(\Gamma)}^2,
\end{equation}
	\begin{equation}
\label{est.4}
\|\bu^\lambda(t) \|_{L^\infty(\Gamma)}\leq C(\|\bu_0\|_{L^1(\Gamma)}, \|\bu_0\|_{L^\infty(\Gamma)} ) \frac{\lambda}{(1+\lambda^2t)^{1/2}}, \quad \forall t>0,
\end{equation}
\begin{equation}
  \label{est.3}
  \|\partial_t \bu^\lambda\|_{L^2(t_1,t_2; V')}\leq C(t_1, q, \|\bu_0\|_{L^1(\Gamma)}).
\end{equation}

\end{proposition}

\begin{proof}
		The first two estimates are consequences of the results obtained in Theorem \ref{decay.estimates}.
		Using that $\bu_0\in L^1(\Gamma)\cap L^\infty(\Gamma)$ we have
		\[
		\|\bu(t)\|_{L^\infty(\Gamma)}\leq \min \Big\{ \|\bu_0\|_{L^\infty(\Gamma)}, t^{-1/2}\|\bu_0\|_{L^1(\Gamma)} \Big\}\leq \frac{C(\|\bu_0\|_{L^\infty(\Gamma)}, \|\bu_0\|_{L^1(\Gamma)})  }
		{(t+1)^{1/2}}.
		\]
		Then the definition of $\bu^\lambda$ gives us the desired estimate.
		
		For the last one we remark that for any $\bpsi \in V$ we have
	\begin{align}
	\label{est.v'}	|\langle \partial_t \bu^\lambda,\bpsi \rangle_{V',V}|&\leq |(\partial_x \bu^\lambda, \partial_x\bpsi)_{L^2(\Gamma)}|+\lambda^{2-q} |(f(\bu^\lambda),\partial_x\bpsi)_{L^2(\Gamma)}|
		 	\\
		\nonumber 	&\leq \|\partial_x \bu^\lambda\|_{L^2(\Gamma)} \|\partial_x\bpsi\|_{L^2(\Gamma)} +\lambda^{2-q} \|f(\bu^\lambda)\|_{L^2(\Gamma)} \|\partial_x\bpsi\|_{L^2(\Gamma)} .
	\end{align}
	This implies that
	\[
	\|\partial_t \bu^\lambda(t)\|_{V'}\leq \|\partial_x \bu^\lambda(t)\|_{L^2(\Gamma)} +\lambda^{2-q}\|\bu^\lambda(t)\|_{L^{2q}(\Gamma)}^q
	\]
	and since $\lambda>1$ and $q\geq 2$ we get
	\begin{align*}
		\|\partial_t \bu^\lambda\|_{L^2((t_1,t_2);V')} &\lesssim \|\partial_x \bu^\lambda\|_{L^2((t_1,t_2); L^2(\Gamma))} + \lambda^{ 2-q}
		\|\bu^\lambda\|_{L^{2q}((t_1,t_2); L^{2q}(\Gamma))}^{ q}\\
		&\leq C (t_1,q,\|\bu_0\|_{L^1(\Gamma)}).
	\end{align*}
This finishes the proof.
\end{proof}

The previous estimates guarantee the compactness of the family $(\bu^\lambda)_{\lambda>1}$.
We now go back to the proof of the main result.

\begin{proof}[Proof of Theorem \ref{first.term}]
We  remark that when $p=1$ property \eqref{q>2}/\eqref{q=2} is equivalent to the existence of a time $t_0>0$ such that
 $\bu_\lambda(t_0)\rightarrow \bu_M(t_0)$ in $L^1(\Gamma)$ as $\lambda\rightarrow\infty$ since
 $$\|\bu_\lambda(t_0)-\bu_M(t_0)\|_{L^1(\Gamma)}=\|\bu(\lambda^2 t_0)-\bu_M(\lambda^2 t_0)\|_{L^1(\Gamma)}.$$
 In the following we will show the convergence of $\bu_\lambda$ toward a function ${\bf w}$ that will be identified later.

\textit{Step I. Compactness}.
Let us consider   the graph $\Gamma_R=\Gamma\cap \{|x|<R\}$ obtained by truncated the graph $\Gamma$. It has all the edges of finite lenght.  
Since $(\bu^\lambda)_{\lambda>1}$ is uniformly bounded in   $L^2_{loc}((0,\infty);V)$ it is also bounded in $H^1(\Gamma_R)$. The uniform estimate  of $(\partial_t\bu^\lambda)_{\lambda>1}$ in 
$L^2_{loc}((0,\infty); V')$  implies that $(\partial_t\bu^\lambda)_{\lambda>1}$  is uniformly bounded in $L^2_{loc}((0,\infty); (H_0^1(\Gamma_R))')$. Since  $H^1(\Gamma_R)$ is compactly embedded in $L^2(\Gamma_R)$ \cite[Lemma 3.7, p.71]{mugnolo}
by Aubin-Lions compactness criterium we find a limit point ${\bf w}=(w_1,\dots, w_{m+n})$ such that up to a subsequence
$\bu^\lambda\rightarrow {\bf w} \ \text{in}\ L^2_{loc}((0,\infty); L^2 (\Gamma_R))$. By a diagonal argument, up to a subseqnece,  $\bu^\lambda\rightarrow {\bf w} \ \text{in}\ L^2_{loc}((0,\infty); L^2_{loc} (\Gamma))$.
Moreover, $\bu^\lambda \rightharpoonup  {\bf w}$ in $L^2_{loc}((0,\infty); V)$ and $\partial_t\bu^\lambda\rightharpoonup  \partial_t{\bf w}$ in
$L^2_{loc}((0,\infty); V')$. In particular, the limit point ${\bf w}\in C((0,\infty),L^2(\Gamma))$. 

A different argument consisting in applying Aubin-Lions's compactness argument on bounded sets  of each of  the half line $I_k$, $k=1,\dots, m+n$, of the graph has been used in \cite[Proof of Th.~3.1]{rossi2020}.

Let us now concentrate on the nonlinear term. Since $\bu^\lambda\rightarrow {\bf w} \ \text{in}\ L^2_{loc}(0,\infty; L^2_{loc}(\Gamma))$ we have
 $\bu^\lambda(t,x)\rightarrow {\bf w}(t,x)$ for a.e. $t$ and $x$ and
\[
\bu^\lambda\rightarrow {\bf w} \ \text{in}\ L^1_{loc}((0,\infty)\times\Gamma).\]
The  a.e. convergence and the fact that $\bu^\lambda(t)$ satisfies
$
\|\bu^\lambda(t)\|_{L^\infty(\Gamma)}\leq Ct^{-1/2}$
imply that ${\bf w}\in L^\infty_{loc}((0,\infty), L^\infty(\Gamma))$ and for a.e. $t>0$, satisfies the same bound  $\|\bw(t)\|_{L^\infty(\Gamma)}\leq Ct^{-1/2}$.
 In particular,   this  implies that
$
\ f(\bu^\lambda)\rightarrow f({\bf w}) \ \text{in}\ L^2_{loc}(0,\infty; L^2_{loc}(\Gamma))
$
and
\begin{equation}
\label{weak-limit-nonlinear}
   f(\bu^\lambda)\rightharpoonup  f({\bf w}) \ \text{in}\ L^2_{loc}(0,\infty; L^2(\Gamma)).
\end{equation}

\medskip

\textit{Step II. Tail control.}
We now prove that for any $\lambda>1$ on the interval $I_k$ we have
\[
 \int_{I_k,\ |x|>2R}|u^\lambda_k(t,x)|dx \leq   \int_{I_k,\ |x|>R}|u_{k0}(x)|dx  +C\Big(\frac t{R^2}+\frac{t^{1/2}}{R}\Big),
\]
for some constant $C=C(\rho,\|\bu_0\|_{L^1(\Gamma)},\|\bu_0\|_{L^\infty(\Gamma)})$. Since $q\geq 2$ the comparison principle holds and it is sufficient to consider the case of nonnegative solutions.

Using the equation satisfied by $\bu^\lambda$ we obtain that for any $\bpsi\in D(\Delta_\Gamma)$
\begin{equation}
\label{identity.u.234}
  (\bu^\lambda(t),\bpsi)_{L^2(\Gamma)}=(\bu^\lambda(0),\bpsi)_{L^2(\Gamma)}+ \int_0^t (\bu^\lambda,\partial_{xx}\bpsi)_{L^2(\Gamma)}ds +\lambda^{2-q}\int _0^t (f(\bu^\lambda), \partial_x \bpsi)_{L^2(\Gamma)}ds.
\end{equation}
Let us choose a   function $\psi^N\in C^2(\rr)$ such that $0\leq \rho\leq 1$, $\rho\equiv 1$ for $2\leq |x|\leq N$ and supported in  $\{1<|x|<N\}$ such that $\|\psi^N_x\|_{L^\infty(\rr)}$ and  $\|\psi^N_{xx}\|_{L^\infty(\rr)}$ do  not depend on $N$. Set $\psi_R^N(x)=\psi^N(x/R)$.
For each $k=1,\dots, m+n$ we take $\psi =(0,\dots,0, \psi_k,0\dots, 0)\in D(A)$
with $\psi_k(x)=\psi_R^N(x)$ on  $I_k$. We obtain for some constant $C=C(\|u_0\|_{L^1(\Gamma)},\|u_0\|_{L^\infty(\Gamma)} )$ the following
\begin{align*}
	 \int_{I_k,\ 2R<|x|<NR} u^\lambda_k(t,x) dx& \leq   \int_{I_k,\ |x|>R}u_{k0}(x)dx +C \Big(t \|(\psi_R^N)_{xx}\|_{L^\infty(\rr)}+t^{1/2}\|(\psi_R^N)_{xx}\|_{L^\infty(\rr)} \Big)\\
	 &  \leq   \int_{I_k,\ |x|>R}u_{k0}(x)dx +C \Big(\frac t{R^n} \|\psi^N_{xx}\|_{L^\infty(\rr)}+\frac{t^{1/2}}{R^{1/2}}\|\psi^N_{xx}\|_{L^\infty(\rr)} \Big).
\end{align*}
Letting $N$ to infinity we obtain  on the interval $I_k$ 
\[ \int_{I_k,\ 2R<|x|} u^\lambda_k(t,x) dx \lesssim    \int_{I_k,\ |x|>R}u_{k0}(x)dx +C \Big(\frac t{R^2}+\frac{t^{1/2}}{R}\Big).
\]
The arguments are similar to the one-dimensional case, see \cite[Lemma 3.3]{MR3190994}. We leave the details to the reader.

The above estimates shows that the convergence of $\bu^\lambda$ towards ${\bf w}$ holds not only in $L^1_{loc}((0,\infty)\times \Gamma)$ but also in $L^1_{loc}((0,\infty);L^1( \Gamma))$. This also guarantees that ${\bf w}$ belongs to $L^1_{loc}((0,\infty);L^1( \Gamma))$ and satisfies
\[
\int_{\Gamma} {\bf w}(t,x)dx=M, \ \mbox{a.e.}\ t>0
\]
and a similar tail control for a.e. $t>0$:
\[
 \int_{I_k, |x|>2R}|w_k(t,x)|dxdt\leq   \int_{I_k, |x|>R}|u_{0k}(x)|dxdt + C \Big(\frac t{R^2}+\frac{t^{1/2}}{R}\Big).
\]

\textit{Step III. Equation satisfied by ${\bf w}$.} We recall that $\bu^\lambda$ satisfies for any $\bpsi\in V$
$$\displaystyle\langle\partial_t\bu^\lambda,\bpsi\rangle_{V',V} + (\partial_x \bu^\lambda, \partial_x\bpsi)_{L^2(\Gamma)} = \lambda^{2-q}(f(\bu^\lambda),\partial_x\bpsi)_{L^2(\Gamma)}, \ \mbox{a.e.}\ t>0.$$

Passing to the limit $\lambda\rightarrow\infty$ in the above equation   we  obtain that  ${\bf w}\in  L^\infty_{loc}((0,\infty), L^\infty(\Gamma))$ satisfies
\[
\left\{
\begin{aligned}
&{\bf w} \in L^2_{loc}((0,\infty); V), \quad \partial_t{\bf w}\in L^2_{loc}((0,\infty); V'), \mbox{ and for all }\bpsi\in V,\\
 &\langle\partial_t {\bf w}(t),\bpsi\rangle_{V',V} + (\partial_x {\bf w}(t), \partial_x\bpsi)_{L^2(\Gamma)} = \delta_2^q (f({\bf w}),\partial_x\bpsi)_{L^2(\Gamma)},
\,\,\mbox{a.e. on } \, (0,\infty),\\
\end{aligned}
\right.
\]
%
where $\delta_2^q=1$ if $q=2$ and vanishes otherwise.
In the case $q>2$ we obtain that ${\bf w}$ is a weak solution of the heat equation whereas
when $q=2$ we obtain that $\bf{w}$ is a solution of the Burgers' equation. Since $\bw(t)\in L^1(\Gamma)\cap L^\infty(\Gamma)$ for a.e. $t>0$, solving the above equation for any $t>t_0$ with $\bw(t_0)\in L^1(\Gamma)\cap L^\infty(\Gamma)$ we obtain  that $\bw\in C([t_0,\infty),L^1(\Gamma))$ and finally  $\bw\in C((0,\infty),L^1(\Gamma))$.

Let us now identify the initial datum in the above system.
Using   identity \eqref{identity.u.234}  we obtain for any $\bpsi\in D(\dg)$ that
\begin{align}
\label{gigi}
\nonumber \Big| & (\bu^\lambda(t),\bpsi)_{L^2(\Gamma)}-\int _{\Gamma} \bu_0(x)\bpsi\left(\frac x\lambda\right)dx\Big|\\
\nonumber& \leq \|\bpsi_{xx}\|_{L^2(\Gamma)}\int _0^t  \|\bu^\lambda(s)\|_{L^2(\Gamma)} ds + \lambda^{2-q}\|  \bpsi_x\|_{L^\infty(\Gamma)} \int _0^t \|\bu^{\lambda}(s)\|_{L^q(\Gamma)}^qds\\
& \leq \|\bpsi_{xx}\|_{L^2(\Gamma)} \|\bu_0\|_{L^1(\Gamma)} \int _0^t  s^{-1/4} ds +\lambda^{2-q}\|  \bpsi_x\|_{L^\infty(\Gamma)}  \int _0^t \|\bu^{\lambda}(s)\|_{L^\infty(\Gamma)}^{q-1}\|\bu^{\lambda}(s)\|_{L^1(\Gamma)}{ ds.}
\end{align}
Using \eqref{est.4}  the last term satisfies
\begin{align*}
  \lambda^{2-q} \int _0^t \|\bu^{\lambda}(s)\|_{L^\infty(\Gamma)}^{q-1}\|\bu^{\lambda}(s)\|_{L^1(\Gamma)}ds& \leq \lambda^{2-q} C(\|\bu_0\|_{L^1(\Gamma)}, \|\bu_0\|_{L^\infty(\Gamma)})\int _0^t \left(\frac{\lambda^2}{1+\lambda^2s}\right)^{(q-1)/2}ds\\
& =
 \begin{cases}
 O(\lambda^{-1}),& q>2,\\
t^{1/2},& q=2.
\end{cases}
\end{align*}
Letting $\lambda\rightarrow\infty$ in \eqref{gigi} we obtain that for any smooth function $\bpsi\in  D(\dg)$
 the following holds
\[
|({\bf w}(t),\bpsi)_{L^2(\Gamma)}-M\bpsi(0)|\lesssim (t^{3/4}+t^{1/2} ) \|\bpsi\|_{D(\dg)}.
\]
This implies that
\[
\lim_{t\downarrow 0} ({\bf w}(t),\bpsi)_{L^2(\Gamma)}=M\bpsi(0), \quad \forall\ \bpsi\in D(\dg).
\]
Using the tail control for $\bu^\lambda$ and for ${\bf w}$  we can obtain that the
same limit holds for all  functions $\bpsi=(\psi_k)_{k=1}^{m+n}$ with $\psi_k\in BC(I_k)$, $k=\{1,\dots, n+m\}$ and
satisfying $\psi_i(0)=\psi_j(0):=\bpsi(0)$, $\forall\ i,j\in \{1,\dots,n+m\}$.

\textit{Step IV. Uniqueness and characterization of the limit profile.} The uniqueness result in Section \ref{self-similar} shows that in both cases $q>2$ or $q=2$, we have ${\bf w}=\bu_M$ and that the whole sequence $(\bu_\lambda)_{\lambda>1} $ converges to $\bu_M$ not only   along  a subsequence.
The strong convergence of $\bu_\lambda$ toward $\bu_M$ in $L^1_{loc}((0,\infty)\times L^1(\Gamma))$ shows the existence of a time $t_0$ such that
$\bu_\lambda(t_0)\rightarrow \bu_M(t_0)$ in $L^1(\Gamma)$. This proves \eqref{q>2} for $p=1$. The general case $1\leq  p<\infty$ follows by using the  case $p=1$ and  the interpolation inequality combined with  the $L^{2p}$ decay of both $\bu(t)$ and $\bu_M(t)$ as $t^{-1/2(1-1/(2p))}$,
\begin{align*}
\label{}
  t^{\frac 12(1-\frac 1p)}\| \bu(t)-\bu_M(t)\|_{L^p(\Gamma)}&\leq t^{\frac 12(1-\frac 1p)}\| \bu(t)-\bu_M(t)\|_{L^1(\Gamma)}^{\frac 1{2p-1}} \| \bu(t)-\bu_M(t)\|_{L^{2p}(\Gamma)}^{\frac{2p-2}{2p-1}}\\
  &\lesssim t^{\frac 12(1-\frac 1p)}\| \bu(t)-\bu_M(t)\|_{L^1(\Gamma)}^{\frac 1{2p-1}} (\| \bu(t)\|_{L^{2p}(\Gamma)}+\|\bu_M(t)\|_{L^{2p}(\Gamma)})^{\frac{2p-2}{2p-1}}\\
  & \lesssim t^{\frac 12(1-\frac 1p)}\| \bu(t)-\bu_M(t)\|_{L^1(\Gamma)}^{\frac 1{2p-1}} \left(t^{-\frac 12 (1-\frac{1}{2p})}\right)^{\frac{2p-2}{2p-1}}\\
  &=
\| \bu(t)-\bu_M(t)\|_{L^1(\Gamma)}^{\frac 1{2p-1}}=o(1), \ t\rightarrow\infty.
\end{align*}

The proof is now complete.
\end{proof}

\section{ Appendix: Comments on the semigroup approach}\label{semigroup-approach}

In order to explain the difficulties in using the semigroup approach for this type of problem, in the following we compute explicitly the linear semigroup $S(t)\bp$ and the commutator  $[\partial_x,S(t)\bp]$. A related work in the case of the Schr\"odinger equation is given in \cite{MR3834707}.
We recall that in \cite[Th. 1.2]{MR2679594} the fact that the semigroup commutes with the derivative $\partial_x$ has been used several times.

Explicit computations show that the linear semigroup can be written as follows.
\begin{lemma}
	\label{heat.semigroup.tree} For any $\bp\in L^2(\Gamma)$ the linear semigroup is given by
	\begin{equation}\label{formula.S}
	(S(t)\bp )(x)=S^-(t)\bp+S^+(t) \Big(-\bp+\frac {2}{m+n}\begin{pmatrix}
	J_{n,n} & 0_{n,m} \\
	0_{m,n} & J_{m,m}
	\end{pmatrix}\bp + \frac 2{n+m}  \begin{pmatrix}
	0_{n,n}  & J_{n,m} \\
	J_{m,n} & 0_{m,m}
	\end{pmatrix}\widetilde \bp  \Big),
	\end{equation}
	where $J_{k,l}$ is the matrix having dimension  $k\times l$ with all the entries equal to one and $S^\pm(t)\bp$, $\bp=(\varphi_1,\dots, \varphi_{n+m})^T$, $\varphi_k:I_k\rightarrow\rr$,  being defined  as follows: for $x\in \rr_-^n\times \rr_+^m$
	\[
	(S^\pm(t){\bm{\varphi}})_k(x)=\int _{I_k} G_t(x\pm y)\varphi_k(y)dy, \, k=1,\dots, n+m,
	\]
	and the components of $\widetilde \bp$ are  $\widetilde \varphi_k(y) = \varphi_k(-y)$,  $k=1,\dots, n+m$.
\end{lemma}

\begin{proof}We remark that since the operator $\Delta_\Gamma:D(\Delta_{\Gamma})\rightarrow L^2(\Gamma)$ is maximal {dissipative it is sufficient to obtain the expression of the semi-group} for $\bp\in D(\dg)$.
	When all the segments are parametrized as $I_k=\rr_+$, $k=1,\dots, n+m$, the solution for $\bp\in D(\dg)$ is given by
	\[
	u_k(x,t)=\int_{I_k}(G_t(x-y)-G_t(x+y))\varphi_k(y)dy+\frac 2{m+n}\int_{\rr_+}G_t(x+y)
	\Big(\sum_{j=1}^{m+n} \varphi_j(y)\Big)dy,
	\]
	where $G_t$ is the one-dimensional heat kernel.
	In our case we can use even extensions of the functions defined on $\rr_-$, apply the above formula and then come back to our initial intervals. The solution of the linear case is then given by
	\begin{align*}
	\label{}
	u_k(x)=&\frac 2{m+n}	\Big(\int_{\rr_-}G_t(x+y)
	\sum_{i=1}^{n} \varphi_i(y)dy+\int_{\rr_+}G_t(x-y)
	\sum_{j=n+1}^{n+m} \varphi_j(y)dy\Big)\\
	&+\int_{I_k}(G_t(x-y)-G_t(x+y))\varphi_k(y)dy, \quad x\in I_k=\rr_-, \, k=1,\dots, n,
	\end{align*}
	and
	\begin{align*}
	\label{}
	u_l(x)=&\frac 2{m+n}	\Big(\int_{\rr_-}G_t(x-y)
	\sum_{i=1}^{n} \varphi_i(y)dy+\int_{\rr_+}G_t(x+y)
	\sum_{j=n+1}^{n+m} \varphi_j(y)dy\Big)\\
	&+\int_{I_l}(G_t(x-y)-G_t(x+y))\varphi_l(y)dy, \quad x\in I_l=\rr_+, l=n+1,\dots, n+m.
	\end{align*}
	Writing $\bp=(\bp^-,\bp^+)^T$ with $\bp^-=(\varphi_1,\dots,\varphi_{n}):\rr_-\rightarrow\rr$  and $\bp^+=(\varphi_{n+1},\dots,\varphi_{n+m}):\rr_+\rightarrow\rr$ we obtain the desired result.
\end{proof}

One of the facts that are specific to this type of problems on networks is the fact that the derivative does not commute with the semigroup (see an example in \cite{MR3834707}). We consider the operator  $\partial_x$ as a closed operator acting on $\widetilde H^1(\Gamma)$ or $H^1(\Gamma)$ with values in $L^2(\Gamma)$.

In fact one can prove that for any $\bp\in \widetilde  H^1(\Gamma)$ the following result holds.

\begin{lemma}\label{derivative.semigroup}For any $\bp\in \widetilde H^1(\Gamma)$, $S(t)\bp\in \widetilde H^1(\Gamma)$ and the following holds
	\begin{align*}
	\partial_x(S(t)\bp)=&S(t)\partial_x\bp+2K_t(x) \begin{pmatrix}
	- I_n  & 0 \\
	0 & I_m
	\end{pmatrix}\Bigg[I_{m+n}-\frac 1{m+n}J_{m+n,m+n}\Bigg]\bp(0)\\
	&+2S^+(t)
	\begin{pmatrix}
	I_n-\frac 2{m+n}J_{n,n} & 0 \\
	0 & I_m-\frac 2{m+n}J_{m,m}
	\end{pmatrix}\partial_x \bp
	\end{align*}
	where $K_t(x)={\rm diag}(G_t(x_k))_{k=1}^{n+m}$.
\end{lemma}

\begin{proof}
	Explicit computations of the derivative of $(S^{\pm}\bp)_{k}$ gives us that
	\[
	(\partial_x S^-(t) \bp ) (x)=K_t(x)\begin{pmatrix}
	-  I_n  & 0 \\
	0 &I_m
	\end{pmatrix}
	\bp(0)
	+(S^-(t) \partial_x \bp )(x)
	\]
	and
	\[
	(\partial_x S^+(t) \bp ) (x)=K_t(x)\begin{pmatrix}
	I_n  & 0 \\
	0 &- I_m
	\end{pmatrix}
	\bp(0)
	-(S^+(t) \partial_x \bp )(x).
	\]
	This give us that the right hand side belongs to $L^2(\Gamma)$. In view of \eqref{formula.S} we have $S(t)\bp\in \widetilde H^1(\Gamma)$ and  the 
	desired formula follows.
\end{proof}

In particular, when the continuity assumption at $x=0$ is assumed the above representation can be simplified.

\begin{lemma}
	\label{derivative.semigroup.continuity}
	For any $\bp\in H^1(\Gamma)$, $S(t)\bp\in  H^1(\Gamma)$ and the following holds
	\[
	\partial_x(S(t)\bp)=S(t)\partial_x\bp+2S^+(t)
	\begin{pmatrix}
	I_n-\frac 2{m+n}J_{n,n} & 0 \\
	0 & I_m-\frac 2{m+n}J_{m,m}
	\end{pmatrix}\partial_x \bp.
	\]
\end{lemma}
\begin{proof}
	When $\bp\in H^1(\Gamma)$ we immediately have from the explicit representation  in \eqref{formula.S} that $S(t)\varphi\in C(\Gamma)$. The previos Lemma already gives us that $S(t)\varphi\in \widetilde H^1(\Gamma)$. Thus $S(t)\varphi\in  H^1(\Gamma)$.
	Moreover, in this case,  the  term containing  $\bp(0)$ that  appears   in \eqref{derivative.semigroup}  vanishes:
	\[
	2K_t(x) \begin{pmatrix}
	- I_n  & 0 \\
	0 & I_m
	\end{pmatrix}\Bigg[I_{m+n}-\frac 1{m+n}J_{m+n}\Bigg]\bp(0)=0_{m+n}
	\]
	and then the result follows.
\end{proof}

\subsection*{Acknowledgments} C. M. C. and L. I. were partially supported by  a grant of Ministery of Research and Innovation, CNCS-UEFISCDI, project number PN-III-P1-1.1-TE-2016-2233, within PNCDI III for the period 2017-2020. The work of L.I. in 2021 has not been supported by any grant of CNCS-UEFISCDI.

A. F. P. was partially supported by CNPq (Brazil) and Agence Universitaire de la Francophonie.

J. D. R. was partially supported by CONICET grant PIP GI No 11220150100036CO
(Argentina), by  UBACyT grant 20020160100155BA (Argentina) and by MINCYT PICT-2018-03183 (Argentina).

We want to warmly thank the referee for his/her comments that helped us to improve our manuscript. The authors also thank to D. Mugnolo for clarifying some results used in the revised version of the manuscript.

\bibliographystyle{plain}

\end{document}